\documentclass[a4paper, 12pt]{article}
\usepackage[utf8]{inputenc}
\usepackage[T1]{fontenc}
\usepackage{lmodern}
\usepackage{graphicx}
\usepackage[english]{babel}
\usepackage{subfigure}
\usepackage[bbgreekl]{mathbbol}
\usepackage{amsmath, amsfonts, amssymb, dsfont}
\usepackage{stmaryrd}
\usepackage{mathrsfs}
\usepackage{mathtools}
\usepackage{listings}
\usepackage{subfigure}
\usepackage{floatrow}
\usepackage{amsthm}
\usepackage{textcomp}
\usepackage{faktor}
\usepackage{xfrac}
\usepackage[all]{xy}
\usepackage[colorlinks = True , allcolors = blue]{hyperref}
\usepackage[symbol]{footmisc}

\newcommand{\supp}{\operatorname{supp}}

\newcommand*\diff{\mathop{}\!\mathrm{d}}

\newcommand{\Id}{\operatorname{Id}}
\newcommand{\rk}{\operatorname{rk}}
\newcommand{\ind}{\operatorname{ind}}
\newcommand{\R}{\mathbb R}

\newcommand{\Z}{\mathbb Z}
\newcommand{\CC}{\mathbb C}
\newcommand{\A}{\mathcal{A}}
\newcommand{\PP}{\mathbb P}
\newcommand{\CCC}{\mathscr{C}}
\newcommand{\T}{\mathbb T}
\newcommand{\kk}{\operatorname{KK}}
\newcommand{\hd}{\Omega^{1/2}}

\newcommand{\G}{\mathcal{G}}

\newcommand{\End}{\operatorname{End}}

\newcommand{\m}{\mathcal{M}}
\newcommand{\op}{\operatorname{Op}}

\newcommand{\g}{\mathfrak{g}}

\newcommand{\Sp}{\operatorname{Sp}}

\newcommand{\vphi}{\varphi}

\newcommand{\rr}{\rightrightarrows}

\newcommand{\inj}{\hookrightarrow}
\newcommand{\act}{\curvearrowright}

\newcommand{\ord}{\operatorname{ord}}
\newcommand{\lag}{\langle}
\newcommand{\rag}{\rangle}

\newcommand{\DNC}{\operatorname{DNC}}
\newcommand{\ev}{\operatorname{ev}}
\newcommand{\ttt}{\mathfrak{t}}
\newcommand{\pr}{\operatorname{pr}}

\newcommand{\ho}{\operatorname{Hol}}

\def\dar[#1]{\ar@<2pt>[#1]\ar@<-2pt>[#1]}
\entrymodifiers={!!<0pt,0.7ex>+}

\newcommand{\eq}[1][r]
   {\ar@<-3pt>@{-}[#1]
    \ar@<-1pt>@{}[#1]|<{}="gauche"
    \ar@<+0pt>@{}[#1]|-{}="milieu"
    \ar@<+1pt>@{}[#1]|>{}="droite"
    \ar@/^2pt/@{-}"gauche";"milieu"
    \ar@/_2pt/@{-}"milieu";"droite"}
\everymath{\displaystyle\everymath{}}

\theoremstyle{plain}
\newtheorem{thm}{Theorem}[section]

\newtheorem{lem}[thm]{Lemma}
\newtheorem{prop}[thm]{Proposition}
\newtheorem{cor}{Corollary}[thm]
\newtheorem{ex}[thm]{Example}

\theoremstyle{definition}
\newtheorem{mydef}[thm]{Definition}

\theoremstyle{remark}
\newtheorem{rem}[thm]{Remark}

\DeclareSymbolFontAlphabet{\mathbb}{AMSb}
\DeclareSymbolFontAlphabet{\mathbbl}{bbold}

\newcommand\blfootnote[1]{%
  \begingroup
  \renewcommand\thefootnote{}\footnote{#1}%
  \addtocounter{footnote}{-1}%
  \endgroup
}

\title{A transverse index theorem in the calculus of filtered manifolds}
\author{Clément Cren\footnote{Univ Paris Est Creteil, Univ Gustave Eiffel, CNRS, LAMA UMR8050, F-94010 Creteil, France} \footnote{ORCID 0000-0001-9217-4953}\\ \small{ Email: \href{mailto:clement.cren@u-pec.fr}{clement.cren@u-pec.fr} } }
\date{}

\begin{document}

\maketitle
\blfootnote{\textit{MSC Classification: }Primary: 58J40, 53C12; Secondary: 58H05, 46L80, 19K56}
\begin{abstract}
We use filtrations of the tangent bundle of a manifold starting with an integrable subbundle to define transverse symbols to the corresponding foliation, define a condition of transversally Rockland, and prove that transversally Rockland operators yield a K-homology class. We construct an equivariant KK-class for transversally Rockland transverse symbols, and show a Poincare duality type result linking the class of an operator and its symbol. 
\end{abstract}

\section{Introduction}

Let $M$ be a manifold and $H^1 \subset \cdots \subset H^r = TM$ a filtration of its tangent bundle such that:
$$\forall i, j\geq 1, \left[ \Gamma(H^i),\Gamma(H^j) \right] \subset \Gamma(H^{i+j}).$$
This setting, called a Lie filtration on $M$, gives rise to the filtered pseudodifferential calculus (also called Heisenberg calculus in the historical case where $r=2$ see \cite{FollandSteinEstimate, bealsgreiner}). Vector fields in $\Gamma(H^i)$ are now considered as differential operators of order $i$. We call H-pseudodifferential operators the ones arising from this calculus. The associated symbols (of order zero) are then multipliers of the noncommutative algebra $C^*(T_HM)$, where $T_HM$ is the Lie group bundle integrating the Lie algebra bundle:
$$\ttt_HM = H^1 \oplus \faktor{H^2}{H^1} \oplus \cdots \oplus \faktor{TM}{H^{r-1}}.$$
It is endowed with the Lie bracket induced by the one of vector fields. This calculus has been introduced to study operators that are hypoelliptic but not elliptic, hence replacing the notion of ellipticity in the filtered calculus by the Rockland condition.
In this article we add an integrable subbundle $H^0 \subset H^1$ and define a notion of  transverse symbols to the foliation induced by $H^0$. We define a "restriction in the transverse direction" map for the symbols, and a transversal Rockland condition in the framework of filtered calculus. We show the existence of a transverse index for symbols satisfying this condition, and show a Poincare duality type result between the class of an operator and the one of its symbol. In the philosophy of noncommutative geometry this setup corresponds to a filtration on the space of leaves of the foliation. We will thus pay a special attention to the case of foliations given by fibrations.
The symbols and operators will be considered from the viewpoint of convolution operators on the tangent groupoid that are quasi-homogeneous with respect to the natural $\R^*_+$-action. This is the framework developped in \cite{van_erpyunken}, it allows to bypass symbolic estimates and is more suited to groupoids and operator albgeras techniques. The article is structured as follows: the first part is devoted to the geometric framework, the holonomy actions, and the introduction of the different groupoids that will be used in the subsequent parts. The second part details the restriction of symbols to the transverse part in the setting of filtered calculus which allows to define a "transversally Rockland" condition. Finally we give two constructions of the transverse cycle, one from the symbol and using KK-theoretic tools, and the other one more analytic, similar to the one in \cite{hilsumskandalis}, using a pseudodifferential operator instead of its symbol, and we prove the equality between the two KK-classes, yielding a Poincare duality type result.

\section{Geometric preliminaries}

\subsection{The setting and the holonomy action}

Let $M$ be a manifold. A foliated filtration is a filtration of the tangent bundle, $H^0 \subset H^1 \subset \cdots \subset H^r = TM$,  such that:
$$\forall i, j\geq 0, \left[ \Gamma(H^i),\Gamma(H^j) \right] \subset \Gamma(H^{i+j}).$$
This condition directly implies that $H^0$ is integrable and that $H^1 \subset \cdots \subset H^r = TM$ is a Lie filtration on $M$ as described in the introduction. Moreover, the Lie bracket of vector fields then gives maps:
$$\Gamma\left(\faktor{H^i}{H^{i-1}}\right) \times \Gamma\left(\faktor{H^j}{H^{j-1}}\right) \to \Gamma\left(\faktor{H^{i+j}}{H^{i+j-1}}\right), \ i,j\geq 0.$$
This way the vector bundles: 
$$\ttt_HM = H^1 \oplus \faktor{H^2}{H^1} \oplus \cdots \oplus \faktor{TM}{H^{r-1}},$$
and
$$\ttt_{\sfrac{H}{H^0}}M = \faktor{H^1}{H^0} \oplus \faktor{H^2}{H^1} \oplus \cdots \oplus \faktor{TM}{H^{r-1}},$$
are both endowed with a Lie algebroid structure over $M$ (with trivial anchor). In fact, for $i,j \geq 1$, the bracket is $\CCC^{\infty}(M)$-bilinear and hence gives a structure of (nilpotent) Lie algebra bundles to $\ttt_HM$ and $\ttt_{\sfrac{H}{H^0}}M$ (see \cite{van_erp_2017,choi_2019,mohsen2018deformation} for more details on this construction). Actually, for fixed $x \in M$, $H^0_x \subset \ttt_{H,x}M$ is an ideal and the quotient Lie algebra is $\ttt_{\sfrac{H}{H^0},x}M$.

Using Baker-Campbell-Hausdorff formula we integrate these Lie algebra bundles into nilpotent Lie group bundles $T_HM$, and $T_{\sfrac{H}{H^0}}M$. The quotient map $\ttt_HM \to \ttt_{\sfrac{H}{H^0}}M$ gives a quotient map $T_HM \to T_{\sfrac{H}{H^0}}M$.\\
Another groupoid useful to our study is the holonomy groupoid of the foliation given by $H^0$. Let us denote it by $\ho(H^0)$. We want to show that this groupoid acts on $T_{\sfrac{H}{H^0}}M$, preserving its groupoid structure (i.e. for all $\gamma \in \ho(H^0), \gamma \colon T_{\sfrac{H}{H^0},s(\gamma)}M \to T_{\sfrac{H}{H^0},r(\gamma)}M$ is a group homomorphism). In order to do so, we need the following result:

\begin{lem}Let $G \rr M$ be a Lie groupoid, and $Y \subset M$ a submanifold. Let $\rho$ be the anchor map of the associated algebroid $\A G \to M$. If for all $y \in Y,$  $\rho(\A_yG) \subset T_yY$ then $G\cdot Y \subset Y,$ i.e. $\forall \gamma \in G, s(\gamma) \in Y \Leftrightarrow r(\gamma) \in Y$.
\end{lem}
\begin{proof}
$\rho(\A_yG)$ is the tangent space to the $G$-orbit passing through $y$, hence the condition implies that an orbit passing through $Y$ stays in $Y$, hence the result holds true.
\end{proof}

\begin{cor}Let $F \subset TM$ be a foliation and $H \subset TM$ be an arbitrary subbundle. If $\left[\Gamma(F),\Gamma(H)\right] \subset \Gamma(H)$ then $\faktor{H}{F}$ (the image of $H$ in $\faktor{TM}{F}$) is invariant under the action $\ho(F) \act \faktor{TM}{F}$.
\end{cor}

\begin{proof}
We apply the lemma to $G = \ho(F)\ltimes \faktor{TM}{F} \rr \faktor{TM}{F}$ and $Y = \faktor{H}{F}$. Let $x \in M$, $\xi \in H_x$. Denote by $\bar{\xi}$ the class of $\xi$ modulo $F_x$. We have the exact sequence:
$$\xymatrix{0 \ar[r] & \faktor{TM}{F} \ar[r] & T\left(\faktor{TM}{F}\right) \ar[r]^-{\diff\pi} & TM \ar[r] & 0},$$
where the first copy of $TM$ corresponds to vertical vector fields for $\pi \colon \faktor{TM}{F} \to M$. When restricted to $Y$, we obtain:
$$\xymatrix{0 \ar[r] & \faktor{H}{F} \ar[r] & TY \ar[r]^{\diff\pi} & TM \ar[r] & 0}.$$
The direct image of $\A_{(x,\bar{\xi})}G$ by the anchor map is then:
$$\xymatrix{0 \ar[r] & [F,\overline{\xi}] \ar[r] & \rho(\A_{(x,\bar{\xi})}G) \ar[r]^-{\diff\pi} & F_x \ar[r] & 0.}$$
Here $[F,\overline{\xi}] = \{ \overline{[X,Y]}(x), X \in \Gamma(F), Y \in \Gamma(TM) \bar{Y}(x) = \bar{\xi} \} \subset \faktor{H_x}{F_x}$ because of the condition $\left[\Gamma(F),\Gamma(H)\right] \subset \Gamma(H)$. We are thus under the conditions of the previous lemma and $\faktor{H}{F}$ is $\ho(F)$-invariant.
\end{proof}

The bracket conditions on foliated Lie filtrations, combined with the corollary, gives that $\ho(H^0)$ acts on each $\faktor{H^i}{H^{0}}$, hence on each $\faktor{H^i}{H^{i-1}}$ ($i\geq 0$). The action is compatible with the brackets $\Gamma(\faktor{H^i}{H^{i-1}}) \times \Gamma(\faktor{H^j}{H^{j-1}}) \to \Gamma(\faktor{H^{i+j}}{H^{i+j-1}})$, and hence gives an action $\ho(H^0) \act  \ttt_{\sfrac{H}{H^0}}M$, compatible with the algebroid structure on $\ttt_{\sfrac{H}{H^0}}M$

\begin{cor}The action $\ho(H^0) \act \ttt_{\sfrac{H}{H^0}}M$ lifts to a groupoid action $\ho(H^0)\act T_{\sfrac{H}{H^0}}M$, which preserves the groupoid structure on $T_{\sfrac{H}{H^0}}M$ .
\end{cor}

\begin{proof}
Since $T_{\sfrac{H}{H^0}}M$ is the s-simply connected groupoid integrating $\ttt_{\sfrac{H}{H^0}}M$, we can lift the action of $\ho(H^0)$ to a groupoid preserving action.
\end{proof}

Note that $\ho(H^0)$ only acts on the normal bundle $\faktor{TM}{H^0}$ and not on the whole tangent bundle to $M$. Hence it does not act on $\ttt_{H}M$, nor on $T_HM$.

\begin{ex} Let $H^1 \subset \cdots \subset H^r = TM$ be a Lie filtration. Let $G$ be a Lie group acting on $M$ locally freely and preserving each $H^i$. Then, by differentiation, the foliation $\mathfrak{g}\ltimes M$ induced by the action (which is a regular foliation because the action is locally free) also preserves each $H^i$. Hence $\tilde{H}^0 = \mathfrak{g}\ltimes M$ and $\tilde{H}^i = H^i + \mathfrak{g}\ltimes M$ defines a foliated filtration, if we assume the $\tilde{H}^i$ to be subbundles.
\end{ex}

\begin{rem}In general we cannot assume that $\mathfrak{g}\ltimes M \subset H^1$. For instance, the Reeb foliation on a contact manifold has its leaves transverse to the contact distribution and also preserves it. Moreover, in our setting, we have to add the assumptions that each $\tilde{H}^i$ is a subbundle. This condition could be removed if one worked with filtrations of the Lie algebra of vector fields by $\CCC^{\infty}(M)$ sub-modules as in \cite{AndroulidakisMohsenYunken}.
\end{rem}

\begin{ex}In the same fashion, let $H^1 \subset \cdots \subset H^r = TM$ be a Lie filtration, and $\mathcal{F} \subset TM$ be an arbitrary foliation such that:
$$\forall i \geq 1, \left[\Gamma(\mathcal{F}),\Gamma(H^i)\right] \subset \Gamma(H^i).$$
Then $\tilde{H}^0 = \mathcal{F}$ and $\tilde{H}^i = \mathcal{F}+H^i$ defines a foliated filtration if each $\tilde{H}^i$ is a subbundle.
\end{ex}

\begin{ex}Let $\pi \colon M \to B$ be a fibration (with connected fibers). Assume $B$ is a filtered manifold, let $\bar{H^1} \subset \cdots \subset \bar{H^r} = TB$ be the corresponding Lie filtration. Let $H^0 = \ker(\diff\pi)$ be the subbundle corresponding to the foliation induced by the fibers of $\pi$. Let $H^i = \pi^*\bar{H^i} \subset TM$ for $i\geq 1$. Then $H^0 \subset H^1 \subset \cdots \subset H^r$ is a foliated filtration on $M$. Furthermore, $\pi$ induces a group isomorphism $T_{\sfrac{H}{H^0},x}M \cong T_{\bar{H},\pi(x)}B$ for every $x \in M$, i.e. $T_{\sfrac{H}{H^0}}M = \pi^*T_{\bar{H}}B$.\end{ex}

\begin{ex}Let $n \in \mathbb{N}$ and $\omega \in \Omega^1(M,\R^n)$ a (n-tuple of) differential form(s) of locally constant rank. Let $H^1 = \ker(\omega)$ and assume there exists subbundles $H^i \subset TM$ with $\Gamma(H^{i+1}) = \Gamma(H^i) + \left[\Gamma(H^i),\Gamma(H^1)\right]$. Assume the existence of $r\geq 1$ such that $H^r = TM$ (we can relax this condition, see remark \ref{VersionGroupoide}). Define also $H^0 = \ker(\diff\omega_{|H^1})$. Then $H^0 \subset \cdots \subset H^r = TM$ is a foliated filtration.\\
Indeed, $H^0$ is a foliation: if $X,Y \in \Gamma(H^0)$, then since $H^0 \subset H^1$:
$$\omega([X,Y]) = X\cdot\omega(Y) - Y\cdot\omega(X) - \diff\omega(X,Y) = 0,$$
so $[X,Y] \in \Gamma(H^1)$. Moreover, if $Z \in \Gamma(H^1)$, then the same computation shows that $[X,Z] \in \Gamma(H^1)$:
$$\diff\omega([X,Y],Z) = [X,Y]\cdot \omega(Z) - Z\cdot\omega([X,Y]) - \omega([[X,Y],Z]) = 0,$$
since by Jacobi's identity: $[[X,Y],Z] = - [[Y,Z],X] - [[Z,X],Y] \in \Gamma(H^1)$.\\
Thus far, we have shown that $H^0$ is a foliation and that it acts on $H^1$. The action on $H^i, i\geq 2$ is then easily proved by induction.
\end{ex}

\subsection{Deformation groupoids}\label{Deformation}

Following the construction of Connes' tangent groupoid (see \cite{connes1994noncommutative}), and due to its interest in pseudodifferential calculus and index theory (see e.g. \cite{debord2014adiab,DEBORD2018255,debord2019lie,van_erp_2017}), a deformation groupoid taking into account the filtered structure of $T_HM$ was constructed in \cite{choi_2019,mohsen2018deformation,van_erp_2017}. We do not recall the construction here but we still construct the associated algebroid through its module of sections:
$$\Gamma(\mathbbl{t}_HM) = \{X \in \Gamma(TM\times \R) \ / \ X_{|t=0} = 0 \ \text{ and } \ \forall i\geq 1, \ \partial^i_tX_{|t=0} \in \Gamma(H^i)\}.$$
The algebroid on $M\times \R$ thus obtained is almost injective and hence integrable, thanks to a theorem of Debord (see \cite{debord2000groupoides}), into an s-connected groupoid $\T_HM$ (which is minimal in some sense to ensure the uniqueness of such groupoid). The references listed above give explicit constructions of this deformation groupoid.
The groupoid $T_{\sfrac{H}{H^0}}M$, on the other hand, does not have a "good" deformation groupoid as it is not directly constructed from a groupoid with filtered algebroid. We can however deform its crossed product with the holonomy groupoid. In order to do this, we give another description of the algebroid of $\ho(H^0)\ltimes T_{\sfrac{H}{H^0}}M  $. \\
Let $\ttt_H^{hol}M = H^0 \oplus \faktor{H^1}{H^0} \oplus \cdots \oplus \faktor{TM}{H^{r-1}}$ be the algebroid whose anchor is given by the injection $H^0 \to TM$, and bracket of sections given by the same kind of formula as for $\ttt_HM$. We then have $ \ttt_H^{hol}M = H^0 \ltimes \ttt_{\sfrac{H}{H^0}}M $ and because $\ho(H^0)$ acts on $\ttt_{\sfrac{H}{H^0}}M$, we can adapt Nistor's proof of integration of crossed product algebroids from \cite{nistor2000groupoids} and integrate $\ttt_H^{hol}M$ into $ \ho(H^0)\ltimes T_{\sfrac{H}{H^0}}M$.
Let us now define the deformation algebroid through its module of sections:
$$\Gamma(\mathbbl{t}^{hol}_HM) = \{X \in \Gamma(TM\times \R) \ / \ \forall i\geq 0, \ \partial^i_tX_{|t=0} \in \Gamma(H^i)\}.$$
It is a subalgebroid of the algebroid of vector fields on $M \times \R$, hence a singular foliation. This foliation is almost regular, the anchor being injective exept for $t = 0$. We can thus construct, by Debord's result, its holonomy groupoid $\mathbb{T}_H^{hol}M$. It is a smooth groupoid over $M\times \R$ that integrates $\mathbbl{t}^{hol}_HM$ and that algebraically satisfies:
$$\mathbb{T}_H^{hol}M = M\times M \times \R^* \sqcup \ho(H^0)\ltimes T_{\sfrac{H}{H^0}}M \times \{0\}.$$
Another construction of $\T_H^{hol}M$ which details the smooth structure is given in \cite{mohsen2018deformation}, example 5.4.

\section{Symbolic calculus}

\subsection{Symbols in the filtered calculus}\label{SectionSymboles}

Let us recall the definition of symbols in the context of filtered calculus. In the following, $\pi \colon G \to M$ denotes a bundle of Carnot Lie groups (i.e. the fibers $G_x, x \in M$ are simply connected nilpotent Lie groups with Lie algebras $\mathfrak{g}_x$ having a decomposition $\mathfrak{g}_x = \oplus_{i = 1}^r \mathfrak{g}_{x,i}$ satisfying $\forall i,j, \ [\mathfrak{g}_{x,i},\mathfrak{g}_{x,j}] \subset \mathfrak{g}_{x,i+j}$). In our setting, we will consider $G = T_HM$ or $G = T_{\sfrac{H}{H^0}}M$. The filtration on the Lie algebra of a Carnot Lie group gives rise to a one parameter subgroup of automorphisms, $\delta_{\lambda} \colon G \to G$, given at the level of the Lie algebra by $\forall i, \diff \delta_{\lambda | \mathfrak{g}_i} = \lambda^i \Id$.

Let $\mathcal{G}$ be a Lie groupoid. In order to write integrals on groupoids and get a convolution algebra, we need the bundle of half densities $\hd$ on $\mathcal{G}$. Replacing $\CCC^{\infty}_c(\mathcal{G})$ by $\Gamma_c(\mathcal{G},\hd)$, we can write formulas like $f\ast g (\gamma) = \int_{\gamma_1\gamma_2 = \gamma} f(\gamma_1)g(\gamma_2)$. We will hence omit the bundle when writing $\CCC^{\infty}_c(\mathcal{G})$ instead of $\Gamma_c(\mathcal{G},\hd)$. However, we will write it if any other bundle is involved (like $\Gamma(\mathcal{G},E\otimes \hd)$ if $E$ is a vector bundle over $\mathcal{G}$). Equivalently, one can also replace half-densities by a choice of a Haar system for $\mathcal{G}$, which would correspond to a trivialization of $\hd$, both yield the same convolution algebra for $\mathcal{G}$. \\
For distributions, $\mathcal{E}'(\mathcal{G})$ will denote the topological dual of $\Gamma(\mathcal{G},\hd)$. The use of half-densities over 'regular' 1-densities or functions is then convenient because we avoid any choice of Haar system, and still have (with our notations) $\CCC^{\infty}_c(\mathcal{G}) \inj \mathcal{E}'(G)$. Indeed for $\alpha \in [0;1]$ we have $(\Omega^{\alpha})^* = \Omega^{1-\alpha}$.

\begin{mydef}A symbol of order $m \in \Z$ in $G$ is a distribution $u \in \mathcal{D}'(G)$ satisfying:
\begin{itemize}
\item $u$ is properly supported, i.e. $\pi \colon \supp(u) \to M$ is a proper map.
\item $u$ is transversal to $\pi$ (in the sense of \cite{androulidakis2009holonomy}).
\item $\forall \lambda \in \R^*_+, \delta_{\lambda *}u -\lambda^m u \in \CCC^{\infty}_p(G)$.
\end{itemize}
$S_p^m(G)$ denotes the set of these distributions, and $S^*_p(G) = \bigcup_{m \in \Z}S_p^m(G)$.
The subspace of compactly supported symbols will be denoted by $S_c^m(G)$. We will write $S^m(G)$ when the statements apply for both $S^m_p(G)$ and $S^m_c(G)$.
\end{mydef}

The second condition allows us to "disintegrate" $u$ along the fibers to obtain a $C^{\infty}$ family of distributions $u_x \in \mathcal{D}'(G_x)$ (see \cite{lescuremanchonvassout} for a more precise statement and more details on transversality of distributions). The first condition then implies that the support of each $u_x$ is a compact subset of the fiber $G_x$ for every $x \in M$, i.e. $u_x \in \mathcal{E}'(G_x)$.
The third condition implies that the singularities of $u$ are located on the zero section $M \subset G$, and also gives the asymptotic behavior of $u$ near $M$.\\
The symbols used here being quasi-homogeneous, they will correspond to principal symbols of H-pseudodifferential operators up to $\CCC^{\infty}_p(G)$ functions. For compactly supported symbols, the quasi-homogeneity condition can equivalently be stated $\mod \CCC^{\infty}_c(G)$.\\
Note that here, the push-forward of distributions is defined by duality with the pullback of half-densities. Let $\lambda > 0$ and $f \in \CCC^{\infty}_c(G)$, we have: 
$$\delta_{\lambda}^*f = \lambda^{n/2}f \circ \delta_{\lambda},$$ 
where $n$ denotes the homogeneous dimension of $G$, i.e. $n = \sum_{i\geq 1} i \rk(\g_i)$. A quick computation shows that, similarly:
$$\delta_{\lambda *}f = \lambda^{n/2} f \circ \delta_{\lambda^{-1}} = \lambda^n \delta_{\lambda^{-1}}^*f.$$
These morphisms define actions of $\R^*_+$ on $\CCC^{\infty}(G), \CCC^{\infty}_c(G)$ and their respective duals. They are compatible with the convolution product for compactly supported functions and distributions.
\begin{rem} As we use properly supported distributions, we do the same for functions: $\CCC^{\infty}_p(\cdot)$ denotes the space of properly supported functions on a groupoid (if $G$ is a groupoid, $X\subset G$ is proper if $s_{|X}$ and $r_{|X}$ are proper maps). With these notations, we see that $\CCC^{\infty}_p(G)$ embeds into the space of properly supported fibered distributions on $G$ (an arbitrary Lie groupoid).\\
Although we will mainly use compactly supported symbols because of their analytic properties (see theorem \ref{OpSymbols} below), we introduced properly supported distributions as they naturally arise when considering deformation groupoids in the framework of \cite{van_erpyunken}.
\end{rem}

\begin{rem}We will also use symbols acting on vector bundles. Let $E,F$ be vector bundles over $M$. We write $S^0(G;E,F)$ for the space of distributions with values in $\End(E,F)$. This means that for every $x \in M$, $u_x \in \mathcal{D}'(G_x) \otimes E_x^* \otimes F_x$. We also write $S^0(G;E) = S^0(G;E,E)$. In order to keep lighter notations, the general results for symbols in the filtered calculus will be stated without vector bundles (note that we need $E=F$ in order to use the convolution product below).
\end{rem}

We recall the useful results on filtered calculus, for more details and proofs we refer to \cite{van_erpyunken, mohsen2020index} and their references.

\begin{lem}Let $u\in S^{m_1}(G), v\in S^{m_2}(G)$ then $u\ast v \in S^{m_1+m_2}(G)$ and $u^* \in S^{m_1}(G)$. This result can be extended to $m_1 = -\infty$ or $m_2 = -\infty$ with $S^{-\infty}_{c/p}(G) = \CCC^{\infty}_{c/p}(G)$.\end{lem}

Let $u \in S^*(G)$, for each $x \in M$ we get a convolution operator 
$$\begin{matrix}
\op(u_x) & \colon & \CCC^{\infty}_c(G_x) & \longrightarrow     & \CCC^{\infty}_c(G_x) \\
         &        &  f                   & \longmapsto & u_x \ast f &
\end{matrix}.$$
The transversality condition then allows us to "glue" these operators to obtain\\
$\op(u) \colon \CCC^{\infty}_c(G) \to \CCC^{\infty}_c(G)$ with $\op(u)(f)_{|\pi^{-1}(x)} = \op(u_x)(f_{|\pi^{-1}(x)})$.

\begin{lem}Let $u,v \in S^*(G)$ and $f,g \in \CCC^{\infty}_c(G)$ then:
\begin{itemize}
\item $\op(u)(f\ast g) = \op(u)(f)\ast g$.
\item $\op(u\ast v) = \op(u) \circ \op(v)$.
\end{itemize}
\end{lem}

\begin{lem}Let $G_1,G_2$ be bundles of Carnot Lie groups over $M$, and $\phi \colon G_1 \to G_2$ a submersive homomorphism of groups bundle which is graduation preserving. $\phi$ induces a map $\phi_* \colon S^*(G_1) \to S^*(G_2)$ obtained by duality with the pullback of functions. It integrates a distribution along the fibers of $\phi$. This map preserves:
\begin{itemize}
\item The order.
\item The sum of elements of the same degree.
\item The adjoint.
\item The convolution product.
\end{itemize}
\end{lem}
\begin{proof}
To show that $\phi_*$ indeed maps $S^*(G_1)$ to $S^*(G_2)$, we need to show that it preserves transversality and quasi-homogeneity. For transversality, it is obvious since $\phi$ is a bundle map. Moreover, since $\phi$ is a morphism between Carnot Lie groups bundles, then it is equivariant with respect to the respective $\R^*_+$-actions on $G_1$ and $G_2$. The conservation of the quasi-homogeneity thus reduces to the fact that $\phi_*(\CCC^{\infty}_p(G_1)) \subset \CCC^{\infty}_p(G_2)$. The proof is then very similar to the one of lemmas \ref{Support} and \ref{QuasiHom} so we refer to these further proofs.\\
The sum, convolution product, and adjoint are then clearly preserved by $\phi_*$.
\end{proof}

We now state the Rockland condition and its consequences on the symbolic calculus. Rockland condition replaces the usual ellipticity condition. Moreover, when the symbols yield operators on $M$ (i.e. for $G = T_HM$, see \cite{van_erpyunken}), the operators whose symbols satisfy Rockland's condition are hypoelliptic and admit parametrices (in the filtered calculus). This condition was first introduced by Rockland in \cite{Rockland} for differential operators on Heisenberg groups. It was then extended to manifolds with osculating Lie group of rank at most two by Helffer and Nourrigat in \cite{HelfferNourrigat}, and Lie groups with dilations in \cite{Groupsdilation}. See also the recent advances in \cite{AndroulidakisMohsenYunken}, in a very broad setting, generalizing the results of Helffer and Nourrigat.

\begin{mydef}A symbol $u \in S^*(G)$ satisfies the Rockland condition at $x \in M$ if there exist a compact set $K \subset \widehat{G_x}$ on the dual space of irreducible representations such that, for all $\pi \notin K$, $\pi(\op(u))$ and $\pi(\op(u^*))$ are injective. The symbol satisfies the Rockland condition if it satisfies it at every point $x \in M$ ($u$ is then also said to be Rockland).
\end{mydef}

\begin{ex} If $H^1 = TM$, $T_HM$ and $T_{\sfrac{H}{H^0}}M$ are abelian group bundles, hence $\widehat{T_HM}_x = T_x^*M$ and $\widehat{T_{\sfrac{H}{H^0}}M}_x = (H^0)^{\perp}$. Therefore, the Rockland condition for symbols on $TM$ and $\faktor{TM}{H^0}$ corresponds respectively to the classical notions of ellipticity and transverse ellipticity. \end{ex}

As in the classical case, Rockland condition corresponds exactly to the existence of parametrices. It was proved in \cite{Groupsdilation} for trivial bundles and this result was used in \cite{dave2017graded} for the proof in the general case.

\begin{thm}[\cite{dave2017graded}]\label{RocklandEq}Let $u \in S^n_p(G)$ then there is an equivalence:
\begin{itemize}
\item[1] $u$ satisfies the Rockland condition.
\item[2] There exists $v \in S^{-k}_p(G)$ such that $u\ast v -1, v\ast u -1 \in \CCC^{\infty}_p(G)$.
\end{itemize}
\end{thm}
A proof can be found in \cite{dave2017graded}, using arguments originating in \cite{Groupsdilation} (which corresponds to the case of trivial bundles).

\begin{thm}[\cite{dave2017graded}]\label{OpSymbols}Let $u \in S^n_c(G)$ then:
\begin{itemize}
\item If $n\leq 0$ then $\op(u)$ extends to an element of the multiplier algebra $\m(C^*(G))$.
\item If $n<0$ then $\op(u)$ extends to an element of $C^*(G)$.
\item If $n>0$, $M$ is compact, and $u$ satisfies the Rockland condition, then $\op(u)$ extends to an unbounded regular operator $\overline{\op(u)}$ on $C^*(G)$ (viewed as a $C^*$-module over itself), and $\overline{\op(u)}^* = \overline{\op(u^*)}$.
\end{itemize}
\end{thm}
Once again see \cite{dave2017graded} for a detailed proof.

\begin{mydef}A symbol $u \in S^*(T_HM)$ is Rockland (i.e "elliptic" in the filtered calculus) if it satisfies the Rockland condition. It is transversally Rockland if the image of $u$ under the natural push-forward map $S^*(T_HM) \to S^*(T_{\sfrac{H}{H^0}}M)$ satisfies the Rockland condition.
\end{mydef}

Theorem \ref{OpSymbols} allows to extend the algebra of symbols. Let $\bar{S}^0_0(G) \subset \m(C^*(G))$ be the $C^*$-closure of $S^0_c(G)$ (we identify the algebra of order 0 symbols and its image by $\overline{\op}$). The algebra $\bar{S}^0_0(G)$ corresponds to the algebra of symbols vanishing at infinity. If $M$ is non-compact, such symbols cannot satisfy the Rockland condition and have a convolution inverse in $\bar{S}^0_0(G)$. We thus need to extend the algebra of symbols to the one of symbols "bounded" at infinity. More precisely, following \cite{hilsumskandalis}, let
$$C^*_M(G) = \{T \in \m(C^*(G)) \ / \ \forall f \in \CCC_0(M), \  fT, Tf \in C^*(G) \},$$
$$\bar{S}^0(G) = \{T \in \m(C^*(G) \ / \ \forall f \in \CCC_0(M), \  fT, Tf \in \bar{S}^0_0(G) \}.$$

We will also need the corresponding algebras of principal symbols :
$$\Sigma^m_c(G) = \faktor{S^m_c(G)}{\CCC^{\infty}_c(G)}, \ \Sigma^0_0(G) = \faktor{\bar{S}^0_0(G)}{C^*(G)}, \ \Sigma^0(G) = \faktor{\bar{S}^0(G)}{C^*_M(G)}.$$
Note that, according to the second criterion in theorem \ref{RocklandEq}, the Rockland condition only depends on the class of a symbol modulo $\CCC^{\infty}_p(G)$, and is thus an invertibility criterion in $\Sigma^0(G)$.
\begin{ex}If $G = G_0 \times M$ is a trivial bundle then $C^*(G) = \CCC_0(M) \otimes C^*(G_0)$, and $C^*_M(G) = \CCC_b(M) \otimes C^*(G_0)$.\\
Every symbol in $S^0_p(G)$ that extends continuously to a multiplier of $C^*(G)$ (e.g. multiplication by functions in $\CCC^{\infty}_p(G)$) defines an element of $\bar{S}^0(G)$.
\end{ex} 

\subsection{Restriction of the symbol}

Let $G_1, G_2$ be arbitrary Lie groupoids with same objects, and $\pi \colon G_1 \to G_2$ a surjective homomorphism (over the identity). Integration along the kernel of $\pi$ gives a *-homomorphism $\int_{\ker(\pi)} \colon C^*(G_1) \to C^*(G_2)$ (here the groupoid $C^*$-algebra is the maximal one, the reduced can be used if $\ker(\pi)$ is amenable\footnote{The trivial representation of $\ker(\pi)$ is then weakly contained in the regular, by continuity of induction then the regular representation of $G_2$ is weakly contained in the regular representation of $G_1$ hence the continuity for the morphism between the reduced $C^*$-algebras}). Here, we use this fact with the morphism $T_HM \to T_{\sfrac{H}{H^0}}M$ and we obtain the surjective morphism of restriction to the transverse directions:
$$\int_{H^0} \colon C^*(T_HM) \to C^*(T_{\sfrac{H}{H^0}}M).$$
By surjectivity of $\pi$, it also extends to the multiplier algebras:
$$\int_{H^0} \colon \m(C^*(T_HM)) \to \m(C^*(T_{\sfrac{H}{H^0}}M)).$$
This morphism sends principal symbols of H-pseudodifferential operators to their restriction to the directions transverse to $H^0$.

\begin{ex}If $H^1 = TM$ we have $T_HM = TM$ and $T_{\sfrac{H}{H^0}}M = \faktor{TM}{H^0}$. Under Fourier transform, we have isomorphisms $C^*(TM) \cong \CCC_0(T^*M)$ and \\
$C^*\left(\faktor{TM}{H^0}\right)\cong \CCC_0((H^0)^{\perp})$. We then have the commutative diagram:
$$\xymatrix{C^*(TM) \ar[r]^{\int_{H^0}} \eq[d] &  C^*\left(\faktor{TM}{H^0}\right) \eq[d] \\
			\CCC_0(T^*M) \ar[r] & \CCC_0((H^0)^{\perp}),}$$
where the bottom arrow corresponds to the restriction of functions on $T^*M$ to $(H^0)^{\perp}\subset T^*M$. This example justifies the name of "restriction" for the morphism $\int_{H^0}$. 
\end{ex}

Using the functoriality of symbols and the integration on $H^0$, we get a commutative diagram:

$$\xymatrix{\bar{S}^0_0(T_HM) \ar[r] \ar@{^{(}->}[d] & \bar{S}^0_0(T_{\sfrac{H}{H^0}}M) \ar@{^{(}->}[d] \\
              \m(C^*(T_HM)) \ar[r]^{\int_{H^0}} & \m(C^*(T_{\sfrac{H}{H^0}}M))}.$$
Therefore, we will also call $\int_{H^0}$ the natural push-forward map $S^*(T_HM) \to S^*(T_{\sfrac{H}{H^0}}M)$ (this diagram also exists for $\bar{S}^0$ on non-compact manifolds).

\subsection{Equivariance of distributions}

In the following sequence we will consider the action of the holonomy groupoid of a foliation on symbols. Let $\G \rr M$ be a Lie groupoid, $E$ be a vector bundle, $G \to M$ a bundle of Carnot Lie groups. Assume $\G$ acts on $G$ and $E$, i.e. for all $\gamma \in \G$, $\gamma \colon G_{s(\gamma)} \to G_{r(\gamma)}$ is a group homomorphism and $\gamma \colon E_{s(\gamma)} \to E_{r(\gamma)}$ is a linear map. In both cases we ask these maps to be compatible with the composition and inverse of the groupoid. If we denote by $s^*G, r^*G$ the group bundles over $\G$ obtained by pullback, and similarly $s^*E, r^*E$ for the pullback vector bundles, an action is just a bundle map $s^*G \to r^*G$ (respectively $s^*E \to r^*E$) compatible with the composition and inverse of the groupoid. See \cite{LeGall}, where this idea is extended to a $C^*$-algebraic context.

We want to understand the action of $\G$ on the spaces of symbols $S^m(G;E)$. Let $\sigma \in S^m(G;E)$, consider the distributions $s^*(\sigma), r^*(\sigma)$ defined respectively by $s^*(\sigma)_{\gamma} = \sigma_{s(\gamma)}, r^*(\sigma)_{\gamma} = \sigma_{r(\gamma)}$. We have $s^*(\sigma) \in S^m(s^*G;s^*E)$ and $r^*(\sigma) \in S^m(r^*G;r^*E)$. The action of $\G$ on the space of symbols is defined by the map $\alpha \colon S^m(s^*G;s^*E) \to S^m(r^*G;r^*E)$ defined by $\alpha(u)_{\gamma} = \gamma \circ u_{\gamma} \circ \gamma^{-1}$. A symbol $\sigma \in S^m(G;E)$ is thus equivariant if $\alpha(s^*(\sigma)) = r^*(\sigma)$, meaning that for every $\gamma \in \G$, $\gamma \circ \sigma_{s(\gamma)} = \sigma_{r(\gamma)} \circ \gamma$. It is equivariant modulo smoothing operators if $\alpha(s^*(\sigma)) - r^*(\sigma) \in \Gamma_c(r^*G;\pi^*r^*\End(E)\otimes \Omega^{1/2})$.

Equivariance can also be considered for $\op(\sigma)$. The actions of \(\G\) on $S^m(G;E)$ and $\Gamma_c(G;\pi^*E \otimes \Omega^{1/2})$ are compatible i.e. $\op$ is equivariant. This means that a symbol $\sigma$ is $\G$-equivariant if and only if the associated convolution operator, 
\[\op(\sigma) \colon \Gamma_c(G;\pi^*E\otimes \Omega^{1/2}) \to \Gamma(G;\pi^*E\otimes \Omega^{1/2}),\] 
is $\G$-equivariant.

For $m = 0$, the action of $\G$ on symbols extends by continuity to $\bar{S}^0_0(G;E)$ and preserves the ideal $C^*(G) \otimes_M \End(E)$, thus defining an action on $\Sigma^0_0(G;E)$. The action also extends to the other algebras of "bounded" symbols. We will then say that a symbol is equivariant modulo compact operators if its class of principal symbol is equivariant. This means for a symbol $\sigma \in \bar{S}^0_0(G)$ that $\alpha(s^*(\sigma)) -r^*(\sigma) \in r^*C^*(G)\otimes_M \End(E)$. In particular, an operator invariant modulo smoothing operator is invariant modulo compact operators.

\section{Transverse cycle for transversally Rockland operators}

\subsection{The KK-cycle of the symbol}

\subsubsection{The equivariant KK-cycle}

\label{setting}
In this section, we define the transverse cycle associated to a transversally Rockland symbol. It is a class in the equivariant $\kk$-group $\kk^{\ho(H^0)}_0(\CCC_0(M), C^*(T_{\sfrac{H}{H^0}}M))$ that gives, after some natural operations in $\kk$-theory, a K-homology class in $K^0(C^*(\ho(H^0)))$. The latter corresponds, in the case where the foliation is a fibration, to the class of some quotient operator on the base of the fibration. The K-homology class thus generally corresponds, in the philosophy of noncommutative geometry, to an abstract Rockland operator on the "space of leaves". Let us first describe the transverse cycle.\\
Let $(E,h)$ be a hermitian, $\ho(H^0)$-equivariant, $\faktor{\Z}{2\Z}$-graded vector bundle over $M$. Let $\sigma \in \bar{S}^0(T_{\sfrac{H}{H^0}}M;E)$ be an order zero odd transverse symbol acting on $E$, which is $\ho(H^0)$-invariant modulo compact operators, and satisfies $\sigma^2- 1 \in C^*_M(T_{\sfrac{H}{H^0}}M)\otimes_M \End(E)$ (see the remark below for a justification of this last assumption). The space of compactly supported smooth sections $\Gamma_c(T_{\sfrac{H}{H^0}}M,\pi^*E\otimes \Omega^{1/2})$ is endowed with a $\Gamma_c(T_{\sfrac{H}{H^0}}M,\Omega^{1/2})$ inner product:
$$\forall f,g \in\Gamma_c(T_{\sfrac{H}{H^0}}M,\pi^*E\otimes \Omega^{1/2}), <f,g>(x,\xi) := \int h(f(x,\xi\eta^{-1}),g(x,\eta)),$$
with $\eta \mapsto h(f(x,\xi\eta^{-1}),g(x,\eta))$ defining a density on $ T_{\sfrac{H}{H^0},x}M$.\\
This product is compatible with the right action by convolution of $\Gamma_c(T_{\sfrac{H}{H^0}}M,\Omega^{1/2})$:
$$\forall f \in \Gamma_c(T_{\sfrac{H}{H^0}}M,\pi^*E\otimes \Omega^{1/2}), \forall \gamma \in \Gamma_c(T_{\sfrac{H}{H^0}}M,\Omega^{1/2}), f \cdot \gamma (x,\xi) := \int f(x,\xi\eta^{-1})\gamma(x,\eta).$$
We can then complete $\Gamma_c(T_{\sfrac{H}{H^0}}M,\Omega^{1/2})$ into $C^*(T_{\sfrac{H}{H^0}}M)$. The module $\Gamma_c(T_{\sfrac{H}{H^0}}M,\pi^*E\otimes \Omega^{1/2})$ becomes a right Hilbert pre-$C^*$-module over $C^*(T_{\sfrac{H}{H^0}}M)$. We complete it and denote by $C^*(T_{\sfrac{H}{H^0}}M,\pi^*E)$ or \(\mathscr{E}\) the corresponding $C^*$-module obtained.\\
Let us denote by $\rho \colon C_0(M) \to \mathcal{B}_{C^*(T_{\sfrac{H}{H^0}}M)}(\mathscr{E})$ the morphism induced on smooth sections by:
$$\forall a \in \CCC^{\infty}_c(M), \forall f \in \Gamma_c(T_{\sfrac{H}{H^0}}M,\pi^*E\otimes \Omega^{1/2}), \rho(a)(f) (x,\xi) = a(x)\rho(x,\xi).$$
This action makes $(\mathscr{E},\rho)$ a $(\CCC_0(M), C^*(T_{\sfrac{H}{H^0}}M))$-$C^*$-bimodule. Indeed, $\rho$ is continuous, of norm 1, so it extends from $\CCC_c^{\infty}(M)$ to $\CCC_0(M)$. Let us now define $F = \overline{\op(\sigma)}$.

\begin{thm}\label{KK-cycle-equiv}If $\sigma \in \bar{S}^0(T_{\sfrac{H}{H^0}}M;E)$ is an order zero odd transverse symbol acting on $E$, which is $\ho(H^0)$-invariant modulo compact operators, and satisfies 
\[\sigma^2- 1 \in C^*_M(T_{\sfrac{H}{H^0}}M)\otimes_M \End(E).\] 
Then $(\mathscr{E},\rho,F)$ is a $\ho(H^0)$-equivariant Kasparov $(\CCC_0(M), C^*(T_{\sfrac{H}{H^0}}M))$-bimodule. 
\end{thm}
\begin{proof}
By construction of $\bar{S}^0(T_{\sfrac{H}{H^0}}M)$, we know that $F \in \m(C^*(T_{\sfrac{H}{H^0}}M)) \otimes_M \End(E) = \mathcal{B}_{C^*(T_{\sfrac{H}{H^0}}M)}(\mathscr{E})$, and that the condition $\sigma^2- 1 \in C^*_M(T_{\sfrac{H}{H^0}}M)\otimes_M\End(E)$ implies that $\rho(a)(F^2-1)$ is a compact operator for every $a \in \CCC_0(M)$.\\
Let $a \in \CCC^{\infty}_c(M)$. For a given $x \in M$, $\op(\sigma_x)$ is a linear operator in the fiber and $\rho(a)$, restricted to the fiber, acts as the scalar multiplication by $a(x)$. We thus get that $[\rho(a),F] = 0$. This result extends by continuity to any $a \in \CCC_0(M)$, i.e. $\forall a \in \CCC_0(M), [\rho(a),F] = 0$ (in particular these operators are compact). \\
Finally, since $\sigma$ is $\ho(H^0)$-equivariant modulo compact operators, then $F$ is $\ho(H^0)$-equivariant modulo compact operators. This means that: 
$$\forall a \in r^*\CCC_0(M), \ \rho(a)(\alpha(s^*(F)) - r^*(F)) \in r^*(C^*(T_{\sfrac{H}{H^0}}M)\otimes_M \End(E)).$$
The Kasparov $(\CCC_0(M), C^*(T_{\sfrac{H}{H^0}}M))$-bimodule $(\mathscr{E},\rho,F)$ is thus $\ho(H^0)$-equivariant in the sense of Le Gall \cite{LeGall}.
\end{proof}

\begin{rem}Let $E_0,E_1$ be hermitian bundles on which $\ho(H^0)$ acts by isometries,\\
$\sigma_0 \in \bar{S}^0(T_{\sfrac{H}{H^0}}M,E_0,E_1)$ be a transversally Rockland symbol of order zero, $\ho(H^0)$-equivariant modulo compact operators. As before we complete $\Gamma_c(E_i)$ into Hilbert-$C^*(T_{\sfrac{H}{H^0}}M)$ modules (ungraded this time) $\mathscr{E}_i$ ($i=0,1$). Let $E = E_0 \oplus E_1$ be the orthogonal sum with the natural $\faktor{\Z}{2\Z}$-grading. The completion of its module of section is $\mathscr{E} = \mathscr{E}_0 \oplus \mathscr{E}_1$. The action of $\CCC_0(M)$ is the same as before. Finally let $\sigma_1 \in S^0(T_{\sfrac{H}{H^0}}M,E_1,E_0)$ be a parametrix of $\sigma_0$ i.e. $\sigma_0\ast\sigma_1-1\in C^*_M(T_{\sfrac{H}{H^0}}M)\otimes_M\End(E_1)$ and $\sigma_1\ast\sigma_0-1 \in C^*_M(T_{\sfrac{H}{H^0}}M)\otimes_M\End(E_0)$. Note that $\sigma_1$ is also $\ho(H^0)$-equivariant modulo compact operators.
Let $\sigma = \begin{pmatrix}0 & \sigma_1 \\ \sigma_0 & 0\end{pmatrix} \in \bar{S}^0(T_{\sfrac{H}{H^0}}M,E)$, it is a $\ho(H^0)$-equivariant (modulo compact operators) transverse symbol with $\sigma^2 -1 \in C^*_M(T_{\sfrac{H}{H^0}}M)\otimes_M\End(E)$, so we are back to our previous setting.
\end{rem}

\begin{cor}
If $(\sigma_t)_{t \in [0;1]}$ is a homotopy of symbols in $\bar{S}^0(T_{\sfrac{H}{H^0}}M)$ satisfying 
$$\forall t \in [0;1], \sigma_t^2-1 \in C^*_M(T_{\sfrac{H}{H^0}}M)\otimes_M\End(E),$$
then the corresponding path, $t\mapsto F_t = \overline{\op(\sigma_t)}$, gives an operator homotopy between $(\mathscr{E},\rho,F_0)$ and $(\mathscr{E},\rho,F_1)$.
\end{cor}
\begin{proof}The operator $\op$ is continuous, hence a continuous path of symbols gives a norm continuous path of operators. The additional assumption on the symbols implies that for all $t \in [0;1]$ we have $(\mathscr{E},\rho,F_t) \in \mathds{E}^{\ho(H^0)}(\CCC_0(M),C^*(T_{\sfrac{H}{H^0}}M))$, hence $(F_t)_{t\in[0;1]}$ is an operator homotopy between $(\mathscr{E},\rho,F_0)$ and $(\mathscr{E},\rho,F_1)$.
\end{proof}
\begin{cor}The $\kk$-theory class $\left[(\mathscr{E},\rho,F)\right] \in \kk_0^{\ho(H^0)}(\CCC_0(M),C^*(T_{\sfrac{H}{H^0}}M))$ only depends on the class of $\sigma$ in $\Sigma^0(T_{\sfrac{H}{H^0}}M;E)$.\end{cor}
\begin{proof}
If $f \in C^*_M(T_{\sfrac{H}{H^0}}M)\otimes_M\End(E)$, then $\sigma_t := \sigma + t f$ satisfies the conditions of the previous corollary, giving an operator homotopy between $(\mathscr{E},\rho,\op(\sigma))$ and $(\mathscr{E},\rho,\op(\sigma + f))$.
\end{proof}

\begin{mydef}We denote by $[\sigma] \in \kk^{\ho(H^0)}(\CCC_0(M), C^*(T_{\sfrac{H}{H^0}}M))$ the $\kk$-theory class described above. It is called the equivariant transverse cycle associated to $\sigma \in \Sigma^0(T_{\sfrac{H}{H^0}}M;E)$.
\end{mydef}

\subsubsection{Descent and index}

Let:
$$j_{\ho(H^0)} \colon  \kk_0^{\ho(H^0)}(\CCC_0(M), C^*(T_{\sfrac{H}{H^0}}M)) \to \kk_0(C^*(\ho(H^0)), C^*(\ho(H^0)\ltimes T_{\sfrac{H}{H^0}}M)),$$
be the descent homomorphism. As described in section 2.2, there is a deformation groupoid, $\T^{hol}_HM$, from $\ho(H^0)\ltimes T_{\sfrac{H}{H^0}}M$ to $M \times M$. In the subsequent parts, we will need to use the fact that $\ev_0 \colon C^*(\T_H^{hol}M_{|[0;1]}) \to C^*(\ho(H^0)\ltimes T_{\sfrac{H}{H^0}}M)$ induces an invertible element in KK-theory. This follows from the contractibility of $C^*(M \times M \times (0;1])$ but we also need the existence of an exact sequence in KK-theory associated with:
$$\xymatrix{0 \ar[r] & C^*(M \times M \times (0;1]) \ar[r] & C^*(\T_H^{hol}M_{|[0;1]}) \ar[r] & C^*(\ho(H^0)\ltimes T_{\sfrac{H}{H^0}}M) \ar[r] & 0}.$$
For this, we need amenability-type assumption on $\ho(H^0)\ltimes T_{\sfrac{H}{H^0}}M$ (K-amenability seems to be the weaker assumption but in practice, amenability would be easier to verify). Since $T_{\sfrac{H}{H^0}}M$ is amenable, the assumption of (K)-amenability on $\ho(H^0)$ implies it on $\ho(H^0)\ltimes T_{\sfrac{H}{H^0}}M$, so we make it from now on. This result on exact sequences goes back to \cite{Kasparov}, see also \cite{DEBORD2018255}. Such assumptions on holonomy groupoids can be found in \cite{tubaumconnes}.
Under such assumptions, we get a class $\ind_H^{hol} = [\ev_0]^{-1}\otimes[\ev_1] \in \kk(\ho(H^0)\ltimes T_{\sfrac{H}{H^0}}M , \CC)$, where $\ev_1 \colon C^*(\T_H^{hol}M_{|[0;1]}) \to C^*(M\times M)$. We then obtain a map:
$$j_{\ho(H^0)}(\cdot)\otimes \ind_H^{hol} \colon \kk_0^{\ho(H^0)}(\CCC_0(M), C^*(T_{\sfrac{H}{H^0}}M)) \to \kk_0(C^*(\ho(H^0)),\CC).$$
We can apply this map to the class of a transversally equivariant transversally Rockland symbol to obtain a $K$-homology class. This class should correspond to some pseudodifferential operator. In the next sections, we construct a $K$-homology class from a $H$-pseudodifferential operator whose transverse symbol is $\ho(H^0)$-equivariant and transversally Rockland. We also show the equality between the resulting $K$-homology class and the one obtained from the transverse symbol of the operator. 

\subsubsection{Symbols of positive order}

Operators arising in a geometric context are often of positive order. We thus want to be able to compute a K-theory class from symbols of positive order as well. In order to do this, we use the Baaj-Julg picture of KK-theory \cite{BaajJulg} using unbounded cycles. The Baaj-Julg picture requires regular operators, we need the third part of theorem \ref{OpSymbols} and thus restrict to the compact case. In addition to the previous assumptions on the symbols, we further need to assume our symbols to be self-adjoint, as it is required in the Baaj-Julg setting. We thus want to prove the following:

\begin{thm}Let $M$ be a compact foliated filtered manifold, $E\to M$ be a $\faktor{\Z}{2\Z}$-graded hermitian vector bundle which is $\ho(H^0)$-equivariant. 
Let $\sigma \in S^m(T_{\sfrac{H}{H^0}}M;E)$ be an odd, self-adjoint, Rockland symbol of positive order, and $\ho(H^0)$-equivariant modulo smoothing operators. 
Denote by $D = \overline{\op}(\sigma)$ the corresponding regular operator on $C^*(T_{\sfrac{H}{H^0}}M;\pi^*E)$. Then $(C^*(T_{\sfrac{H}{H^0}}M;\pi^*E),D)$ is an unbounded Kasparov $(\CCC_0(M),C^*(T_{\sfrac{H}{H^0}}M))$-bimodule. 
Its class in $KK^{\ho(H^0)}(\CCC_0(M),C^*(T_{\sfrac{H}{H^0}}M))$ does not depend on the self-adjoint representative of the class $\bar{\sigma} \in \Sigma^m(T_{\sfrac{H}{H^0}}M;E)$.
\end{thm}
\begin{proof}
The fact that $(C^*(T_{\sfrac{H}{H^0}}M;\pi^*E),D)$ is an unbounded Kasparov bimodule is rather straightforward to prove. The fact that $D$ is self-adjoint follows from the assumption that $ \sigma^* = \sigma$, and, as for the bounded case, the commutators with elements of $\CCC_0(M)$ vanish and in particular they are compact. The only thing to show is that $(1+D^2)^{-1} \in \mathcal{K}(C^*(T_{\sfrac{H}{H^0}}M;\pi^*E))$. Let $q \in S^{-2m}(T_{\sfrac{H}{H^0}}M;E)$ be a parametrix for $\sigma$ and denote by $Q := \overline{\op}(q) \in \mathcal{K}(C^*(T_{\sfrac{H}{H^0}}M;\pi^*E))$ the associated operator. Let $\rho = \sigma^2q - 1 \in \Gamma_c(T_{\sfrac{H}{H^0}}M;\pi^*\End(E)\otimes \Omega^{1/2})$ and $R = \overline{\op}(\rho) \in \mathcal{K}(C^*(T_{\sfrac{H}{H^0}}M;\pi^*E))$ the associated operator. Since $\ord(\sigma^2) + \ord(q) = 0 \leq 0$ and $\ord(q) = -2m \leq 0$, it follows that $\overline{\op}(\sigma^2q) = \overline{\op}(\sigma^2)\overline{\op}(q) = D^2Q$. From that we get that $1 = R - D^2Q = R - (1+D^2)Q + Q$, hence $(1+D^2)^{-1} = (1+D^2)^{-1}R - Q + (1+D^2)^{-1}Q$. Since $D$ is regular, we know that $(1+D^2)^{-1}$ is bounded and the previous equation shows its compactness.

Thus far we have obtained an unbounded Kasparov bimodule. We know from the results of Baaj and Julg that its bounded transform gives a (bounded) Kasparov module. We thus need to show that the bounded transform of $D$ is equivariant modulo compact operators. Denote by $D_r = \overline{\op}(r^*\sigma) = r^*D$ and $D_s = s^*D$ the respective pullback operators on $\ho(H^0)$. We know that $k := r^*\sigma -\alpha(s^*\sigma) \in \Gamma_c(r^*T_{\sfrac{H}{H^0}}M,\pi^*r^*\End(E)\otimes \Omega^{1/2})$ and denote by $K \in r^*\mathcal{K}(C^*(T_{\sfrac{H}{H^0}}M;\pi^*E))$ the associated operator. Since $k$ is bounded, we have $D_r - \alpha(D_s) = K$. Let $f \colon T \mapsto T(1+T^*T)^{-1/2}$ denote the bounded transform. We need to show that $\Delta := f(D_r + K) - f(D_r) \in \mathcal{K}_{\ho(H^0)}(r^*C^*(T_{\sfrac{H}{H^0}}M;\pi^*E))$. We will use the following facts :
\begin{itemize}
\item If $Q,S$ are invertible unbounded operators $Q^{-1} - S^{-1} = Q^{-1}(S-Q)S^{-1}$.
\item If $Q$ is a regular operator $(1+Q^*Q)^{-1/2} = \frac{1}{\pi}\int_{0}^{+\infty}\lambda^{-1/2}(1+\lambda + Q^*Q)^{-1}\diff \lambda$ the integral being absolutely convergent.
\end{itemize}
\begin{align*}
\Delta &= f(D_r + K) - f(D_r) \\
       &= (D_r + K) \left[ (1 + (D_r + K)^2)^{-1/2} - (1 + D_r^2)^{-1/2} \right] + K(1+D_r^2)^{-1/2} \\
       &= \frac{1}{\pi}\int_0^{+\infty}\lambda^{-1/2} (D_r + K) \left[ (1 + \lambda + (D_r + K)^2)^{-1} - (1+ \lambda + D_r^2)^{-1} \right] \diff \lambda \\
       & \ \ \ + K(1+D_r^2)^{-1/2} \\
       &= \frac{1}{\pi}\int_0^{+\infty}\lambda^{-1/2} (D_r + K)(1 +\lambda + (D_r + K)^2)^{-1}K'(1 +\lambda + D_r^2)^{-1}\diff \lambda \\
       & \ \ \ + K(1+D_r^2)^{-1/2},
\end{align*}
where $K' = D_r^2 - (D_r + K)^2 = - D_rK - KD_r - K^2 = -\overline{\op}(\sigma_rk + k\sigma_r + k^2)$. Since $k$ is smoothing, then the second summand of the RHS is compact and we only have to prove that the integral part is compact. First, $k$ being smoothing, then so is $\sigma_rk + k\sigma_r + k^2$, and thus $K' \in r^*\mathcal{K}(C^*(T_{\sfrac{H}{H^0}}M;\pi^*E)) \subset \mathcal{K}_{\ho(H^0)}(r^*C^*(T_{\sfrac{H}{H^0}}M;\pi^*E))$. We then have:
$$(D_r + K)(1 +\lambda + (D_r + K)^2)^{-1} = f(D_r + K)(1+(D_r + K)^2)^{-1/2}(1+\lambda(1+(D_r+K)^2)^{-1})^{-1},$$
so $(D_r + K)(1 +\lambda + (D_r + K)^2)^{-1} \in \mathcal{K}_{\ho(H^0)}(r^*C^*(T_{\sfrac{H}{H^0}}M;\pi^*E))$. Finally, the same goes for $(1 +\lambda + D_r^2)^{-1}$. We need to check that the integral converges in order to prove that it defines an element of $\mathcal{K}_{\ho(H^0)}(r^*C^*(T_{\sfrac{H}{H^0}}M;\pi^*E))$. We prove it is absolutely convergent. Using functional calculus, we prove that:
$$\|(D_r + K)(1 +\lambda + (D_r + K)^2)^{-1}\| = \sup_{\mu\in \Sp(D_r + K)} \frac{|\mu|}{1+\lambda + \mu^2} \leq \frac{1}{2\sqrt{1+\lambda}}.$$
Likewise, $\|(1+\lambda + D_r^2)\| \leq \frac{1}{1+\lambda}$. Denote by $\lambda \mapsto T(\lambda)$ the integrand. We have: 
$$\forall \lambda > 0, \|T(\lambda)\| \leq \frac{\|K'\|}{2\sqrt{\lambda}(1+\lambda)^{3/2}}.$$
Thus, the integral converges absolutely and $\Delta \in \mathcal{K}_{\ho(H^0)}(r^*C^*(T_{\sfrac{H}{H^0}}M;\pi^*E))$, thus ending the proof.
\end{proof}

\subsection{The pseudodifferential construction}

In this section, we give a more explicit construction of the cycle representing the K-homology class $j_{\ho(H^0)}([\sigma])\otimes \ind_H^{hol}$ with H-pseudodifferential operators. We first recall their definition, and then follow the reasoning of \cite{hilsumskandalis} appendix A to show that a H-pseudodifferential operator with symbol $\sigma$ (which, as in the previous section, is transversally Rockland) defines a $KK$-cycle representing  
\[j_{\ho(H^0)}([\sigma])\otimes \ind_H^{hol} \in \kk_0(C^*(\ho(H^0)),\CC).\]

\subsubsection{The van Erp Yunken picture of pseudodifferential calculus}

Each deformation groupoid constructed in section \ref{Deformation} is endowed with a $\R^*_+$ action. We describe briefly the one of $\T_HM$ denoted by $\alpha_{\lambda}$ for $\lambda > 0$. We have
\begin{align*}
\alpha_{\lambda}(x,y,t) &= (x,y,\lambda^{-1}t) \\
\alpha_{\lambda}(x,\xi,0) &= (x,\delta_{\lambda}(\xi),0),
\end{align*}
where $\delta_{\lambda}$ denotes the $\R^*_+$-action on $T_HM$ described in section \ref{SectionSymboles}. This is a smooth action by groupoid homomorphism. As for $\delta_{\lambda}$, it induces $\R^*_+$-actions on spaces of functions (denoted by $\alpha_{\lambda}^*$ for $\lambda > 0$) and their dual spaces (denoted by $\alpha_{\lambda *}$ for $\lambda > 0$). Both are compatible with the convolution product of compactly supported functions and distributions.

\begin{mydef}A  H-pseudodifferential operator on $M$ of order $m$ is a distribution $P \in \mathcal{E}'(M\times M)$ such that there exists $\PP = (\PP_t)_{t\in\R} \in \mathcal{D}'(\mathbb{T}_HM))$, properly supported, such that:
\begin{itemize}
\item $\PP$ is transversal to the range map $r$.
\item $\PP$ is quasi homogeneous with respect to the zooming action i.e.\\
$\forall \lambda>0, \alpha_{\lambda*}\PP - \lambda^m\PP \in \CCC^{\infty}_p(\mathbb{T}_HM).$
\item $\PP_1 = P$.
\end{itemize}
We say that $\PP$ is an extension of $P$. The operator $P$ is compactly supported if there exists $K \subset M \times M$ such that $\forall t \in \R^*, \supp(\PP_t) \subset K$. We denote respectively by $\Psi^m_{H,p}(M)$ and $\Psi^m_{H,c}(M)$ the set of properly supported (resp. compactly supported) pseudodifferential operators, and by $\bbpsi^m_{H,p}(M)$ and $\bbpsi^m_{H,c}(M)$ the set of their respective extensions to $\T_HM$.
\end{mydef}

In this definition, due to \cite{van_erpyunken}, pseudodifferential operators are given by their Schwartz kernels and the operators themselves obtained by convolution. The case of compactly supported operators appears in the work of Ewert \cite{Ewert}. Although it is not mentioned in the definition, the quasi-homogeneity condition forces the singular support of a H-pseudodifferential operator to be located on the diagonal of $M \times M$. This definition can also be adapted to H-pseudodifferential operators between vector bundles over $M$ by replacing $\mathcal{D}'(\T_HM)$ with $\mathcal{D}'(\T_HM,\hom(s^*(E\times \R),r^*(F\times \R))\otimes \hd)$ .\\
The definition of compactly supported operators of Ewert might seem restrictive, the following lemma shows that one can only assume compactness of the operator itself.

\begin{lem}Let $P \in \Psi^m_{H,p}(M)$ be an operator which is compactly supported as an operator on M. Then $P \in \Psi^m_{H,c}(M)$.\end{lem}

\begin{proof}
Take $\PP \in \bbpsi^m_{H,p}(M)$ such that $\PP_1 = P$. Let $\vphi_1 \in \CCC^{\infty}_c(M \times M)$ be a bump function with $\vphi_1 \equiv 1$ on $\supp(P)$. Let $\vphi = \vphi_1 \circ (\pr_M \circ s, \pr_M \circ s)$ and $\PP^{\, \prime} = \vphi\PP$. We have $\PP^{\, \prime} \in \bbpsi^m_{H,c}(M)$. Indeed $\PP^{\, \prime}$ is transversal to the range map and $\PP - \PP^{\, \prime} \in \CCC^{\infty}_p(\T_HM)$ hence $\PP^{\, \prime}$ is also quasi-homogeneous of the same degree as $\PP$. Finally by construction of $\vphi$, we obviously have $\PP^{\, \prime}_1 = P$, hence the result.
\end{proof}

If $P \in \Psi^m_H(M)$ and $\PP \in \bbpsi^m_{H,p/c}(M)$ is such that $\PP_1 = P$, then $\PP_0 \in S^m_{p/c}(T_HM)$ as $\alpha_{\lambda|T_HM} = \delta_{\lambda}$. It is called the symbol of $P$. A different choice in $\PP$ would result on a different symbol that will be equal to $\PP_0$ modulo $\CCC^{\infty}_p(T_HM)$ (see \cite{van_erpyunken}). This gives a well defined principal symbol map $\Psi^m_{H,p/c}(M) \to \Sigma^m_{p/c}(T_HM)$. The map $\bbpsi^{m-1}_{H,p/c}(M) \to \bbpsi^{m}_{H,p/c}(M)$ sending $\PP$ to $(t\PP_t)_{t \in \R}$ gives an injection $\Psi^{m-1}_{H,p/c}(M) \to \Psi^{m}_{H,p/c}(M)$. When restricted at $t =0$ we get $0$, therefore the image of this morphism is contained in the kernel of the symbol map. The converse also holds true: if $P \in \Psi^{m}_{H,p/c}(M)$ has a vanishing symbol, then we can take an extension $\PP$ of $P$  with $\PP_0 = 0$. The distribution defined for $t \neq 0$ by $t^{-1}\PP_t$ extends to $\T_HM$ and defines an element of $\bbpsi^{m-1}_{H,p/c}(M)$ (see proposition 37 of \cite{van_erpyunken} for more details). The operator $P$ is hence in the image of $\Psi^{m-1}_{H,p/c}(M)$. We have the exact sequence:
$$\xymatrix{0 \ar[r] & \Psi^{m-1}_{H,p/c}(M) \ar[r] & \Psi^m_{H,p/c}(M) \ar[r] & \Sigma^m_{p/c}(T_HM) \ar[r] & 0}.$$
As for the symbols, elements of $\Psi^m_{H,p/c}(M)$ act by convolution on $\CCC^{\infty}_c(M)$. We will use the following analytic results:

\begin{thm}[\cite{dave2017graded}]Let $P \in \Psi^m_{H,c}(M)$:
\begin{itemize}
\item If $m = 0$ then $P$ extends to a bounded operator on $L^2(M)$.
\item If $m < 1$ then $P$ extends to a compact operator on $L^2(M)$.
\end{itemize}
\end{thm}

We denote by $\Psi^0_{H,0}(M)$ the norm closure of $\Psi^0_{H,c}(M) \subset \mathcal{B}(L^2(M))$. The closure of the space of compactly supported operators of negative order give the whole subspace of compact operators. The previous exact sequence extends by continuity to the exact sequence:
$$\xymatrix{0 \ar[r] & \mathcal{K}(L^2(M)) \ar[r] & \Psi^0_{H,0}(M) \ar[r] & \Sigma^0_{0}(T_HM) \ar[r] & 0}.$$
As before, for the non-compact case, we want to extend our algebra of operators to those "bounded at infinity". Define: 
$$\mathcal{K}_M(L^2(M)) = C^*_M(M\times M) = \{ T \in \mathcal{B}(L^2(M)) \ / \ \forall f \in \CCC_0(M), \ fT,Tf \in \mathcal{K}(L^2(M)) \},$$
and 
$$\Psi^*_H(M) = \{ T \in \mathcal{B}(L^2(M)) \ / \ \forall f \in \CCC_0(M), \ fT,Tf \in \Psi^0_{H,0}(M) \}.$$
We get the exact sequence:
$$\xymatrix{0 \ar[r] & \mathcal{K}_M(L^2(M)) \ar[r] & \Psi^*_{H}(M) \ar[r] & \Sigma^0(T_HM) \ar[r] & 0}.$$

\begin{rem}Note that $\Psi^0_{H,p}(M) \cap \Psi^*_H(M) \subset \Psi^*_H(M)$ is a dense subset. In what follows, we will thus assume the operators in $\Psi^*_H(M)$ to be given by actual H-pseudodifferential operators (with possibly non-compact support but still continuous).\end{rem}

\subsubsection{The K-homology cycle construction}

We now want to construct a K-homology cycle for $C^*(\ho(H^0))$ from an order $0$ H-pseudodifferential operator which is transversally Rockland. The idea is, locally, to induce an operator on some transversal to the foliation. We will get compactness on this transversal using a parametrix. The compactness on the leaves directions will be given by elements of $C^*(\ho(H^0))$, giving a K-homology cycle. Let us first assume that the foliation is trivial (we will reduce to this case by taking foliated charts), i.e. $M = U \times T$, where the foliation is given by $TU$, and $T$ is a filtered manifold with filtration $H'$. In this setting we have:
\begin{align*}
\ho(H^0) &= M \times_T M = U \times U \times T\\
\mathbb{T}_HM &= \mathbb{T}U \times_{\R} \mathbb{T}_{H'}T\\
\mathbb{T}^{hol}_HM &= (U \times U \times \R) \times_{\R} \mathbb{T}_{H'}T = (U \times U) \times \mathbb{T}_{H'}T = \ho(H^0) \times_T \T_{H'}T.
\end{align*}
The canonical morphism $\vphi \colon \mathbb{T}_HM \to \mathbb{T}_H^{hol}M$ is then expressed as the product of the identity of $\mathbb{T}_{H'}T$ and the map $\DNC(\Id_{U\times U}) \colon \mathbb{T}U \to U \times U \times \R$. The map $\DNC(\Id_{U\times U})$ comes from the functoriality of the deformation to the normal cone (DNC) construction (see \cite{mohsen2018deformation}), together with the identifications $\mathbb{T}U = \DNC(U\times U, \Delta_U)$ and $U \times U \times \R = \DNC(U\times U,U\times U)$.

For $\lambda >0$ we write $\alpha_{\lambda}, \alpha_{\lambda}^{hol}, \alpha_{\lambda}^{(U)}, \alpha_{\lambda}^{(T)}$ for the respective actions of $\R^*_+$ on $\T_HM$, $\T^{hol}_HM$, $\T U$ and $\T_{H'}T$. We have the relations $\alpha_{\lambda} = \alpha_{\lambda}^{(U)} \otimes_{\R} \alpha_{\lambda}^{(T)}$ and $\alpha_{\lambda}^{hol} = 1\otimes \alpha_{\lambda}^{(T)}$. \\
Let $E \to M$ be a hermitian $\ho(H^0)$-equivariant vector bundle. The holonomy invariance implies that $E = \pr_T^*(E_2)$ for some hermitian vector bundle $E_2 \to T$. Let $\mathscr{E} = C^*(\T^{hol}_HM, r^*E)$ and $\mathscr{E}_2 = C^*(\T_{H'}T, r^*E_2)$ we have:
$$\mathscr{E} = C^*(U \times U \times \R) \otimes_{\CCC_0(\R)} \mathscr{E}_2 = C^*(\ho(H^0)) \otimes_{\CCC_0(T)}\mathscr{E}_2 = C^*(U \times U) \otimes \mathscr{E}_2.$$
From a distribution $\PP \in \bbpsi^m_{H,p}(M,E)$ we want to induce a distribution $\PP^{\, \prime} \in \bbpsi^m_{H',c}(T,E')$. In the case $m = 0$ we want $P$ to be, in the terminology of \cite{connes1984longitudinal}, a $C^*(\ho(H^0))$-connection for $P'$ (the non-blackboard letters denote the restriction of the respective distributions on $M \times M$ at time 1).\\
Let $\eta,\eta' \in \CCC^{\infty}_c(\ho(H^0))$ and denote by $T_{\eta} \colon \mathscr{E}_2 \to \mathscr{E}$ the (interior) tensor product by $\eta$.

\begin{thm}\label{InducedOP}
Consider $\PP^{\, \prime} = T_{\eta'}^*\vphi_*(\PP)T_{\eta}$ for \(\eta,\eta'\in \CCC^{\infty}_c(\ho(H^0))\). Then we have $\PP^{\, \prime} \in \bbpsi^m_{H',p}(T,E')$ and $P' := \PP^{\, \prime}_1 \in \Psi^m_{H',c}(M,E')$.
\end{thm}

Let us first explain what $\PP^{\, \prime}$ does. Let $f,g \in \Gamma_c(\T_{H'}T,r^*E')$ then:
\begin{align*}
\lag \PP^{\, \prime} f,g \rag &= \lag \vphi_*(\PP) T_{\eta}f, T_{\eta'}g\rag \\
                   &= \lag \vphi_*(\PP)(\eta\otimes f), (\eta'\otimes g)\rag \\
                   &= \lag \vphi_*(\PP), (\eta'\otimes g) \ast (\eta \otimes f )^* \rag \\
                   &= \lag \vphi_*(\PP), (\eta' \ast \eta^*) \otimes (g \ast f^*)\rag  \\
                   &= \lag (\eta \ast \eta'^*\otimes 1) \vphi_*(\PP),  1 \otimes (g \ast f^*)\rag.
\end{align*}
Since $\pr_T^*f = f \circ \pr_T = 1\otimes f$, we get:
$$\lag \PP^{\, \prime} f,g \rag = \lag \pr_{T*}\left((\eta \ast \eta'^*\otimes 1) \vphi_*(\PP)\right)f,  g\rag,$$
hence:
$$\PP^{\, \prime} = \pr_{T*}(((\eta\ast\eta'^*)\otimes 1)\ast \vphi_*(\PP)).$$

\begin{lem}If $\Theta_{\eta,\eta'} = T_{\eta'}^* \cdot T_{\eta}$, then:
$$\forall \lambda >0, \alpha_{\lambda*}^{(T)} \circ \Theta_{\eta,\eta'} = \Theta_{\eta,\eta'} \circ \alpha_{\lambda*}^{hol}.$$
\end{lem}
\begin{proof}
Let $\lambda > 0$, let $\mathbb{Q} \in \mathscr{E}'(\T^{hol}_HM)$. The previous computation shows that:
$$\Theta_{\eta,\eta'}(\mathbb{Q}) =  \pr_{T*}(((\eta\ast\eta'^*)\otimes 1)\ast\mathbb{Q}).$$
Using that $\pr_T$ commutes with the respective $\R^*_+$-actions we obtain:
\begin{align*}
\alpha_{\lambda *}^{(T)}\circ \Theta_{\eta,\eta'}(\mathbb{Q})  &= (\alpha_{\lambda}^{(T)} \circ \pr_T)_*(((\eta\ast\eta'^*)\otimes 1)\ast\mathbb{Q}) \\
															 &= \pr_{T *}(\alpha_{\lambda *}^{hol}((\eta\ast\eta'^*)\otimes 1)\ast \alpha_{\lambda *}^{hol}(\mathbb{Q})) \\
															 &= 	\pr_{T *}(((\eta\ast\eta'^*)\otimes 1)\ast \alpha_{\lambda *}(\mathbb{Q})) \\	
															 &= \Theta_{\eta,\eta'}(\alpha_{\lambda *}^{hol}(\mathbb{Q}),	 
\end{align*}
with $\alpha_{\lambda *}^{hol}((\eta\ast\eta'^*)\otimes 1) = (\eta\ast\eta'^*) \otimes 1$. This is because $\alpha_{\lambda}^{hol} = 1 \otimes \alpha_{\lambda}^{(T)}$ and $1$ is invariant under the zoom action on $T$.
\end{proof}

\begin{rem} This lemma shows the usefulness of $\T^{hol}_HM$. One could have taken $\eta, \eta' \in \CCC^{\infty}_c(\T U)$, define $T_{\eta} = \eta \otimes_{\CCC_0(\R)}\--$, and induce directly $\PP^{\, \prime} = T_{\eta'}^*\PP T_{\eta}$. However, the same calculations would have yielded $\alpha_{\lambda*}^{(T)} \circ \Theta_{\eta,\eta'} = \Theta_{\eta\circ \alpha_{\lambda^{-1}}^{(U)},\eta'\circ \alpha_{\lambda^{-1}}^{(U)}} \circ \alpha_{\lambda*}$, which would then impact the next lemmas. Also, using $\T^{hol}_HM$ makes the holonomy groupoid appear explicitly in local foliated charts.
\end{rem}

\begin{lem}\label{Support}$\vphi \colon \T_HM \to \T_H^{hol}M$ maps proper subsets to proper subsets.\end{lem}
\begin{proof}
Let $X \subset T_HM$ be a proper subset, $K \subset M \times \R$ a compact set. Since $\vphi$ is a groupoid homomorphism we have $r_{|\vphi(X)}^{-1}(K) = r^{-1}(K)\cap \vphi(X) = \vphi(X\cap r^{-1}(K))= \vphi(r^{-1}_{|X}(K))$. Since $X$ is proper $r^{-1}_{|X}(K)$ is compact and so is $r_{|\vphi(X)}^{-1}(K)$. The same goes for $s$, hence $\vphi(X)$ is proper.
\end{proof}

\begin{lem}\label{QuasiHom}For all $\eta,\eta' \in \CCC^{\infty}_c(\ho(H^0))$ we have $\Theta_{\eta,\eta'}\circ \vphi_* (\CCC_p^{\infty}(\T_HM)) \subset \CCC_p^{\infty}(\T_{H'}T)$.\end{lem}
\begin{proof}
Let $f \in \CCC^{\infty}_p(\T_HM)$, recall that $\Theta_{\eta,\eta'}(\vphi_*(f)) = \pr_{T*}(((\eta\ast\eta')\otimes 1) \vphi_*(f))$ hence $\supp(\Theta_{\eta,\eta'}(\vphi_*(f))) = \pr_T(\supp(\eta\ast \eta'^* \otimes 1) \cap \supp(\vphi_*f))$. We have the inclusion 
\[\supp(\vphi_*(f)) \subset \vphi(\supp(f)).\]
Lemma \ref{Support} implies that $\vphi(\supp(f))$ is proper. Since $\supp(\vphi_*(f))$ is a closed subset contained in a proper subset, it is proper as well. Finally, $\eta$ and $\eta'^*$ being compactly supported we may take $K \subset U$ compact such that $\supp(\eta\ast \eta'^* \otimes 1) \subset K \times K \times \T_{H'}T$. Let $C \subset T \times \R$ be a compact set:
\begin{align*}
r_{|\supp(\Theta_{\eta,\eta'}(\vphi_*(f)))}^{-1}(C) &\subset \pr_T(\pr_T^{-1}(r^{-1}(C)) \cap (K^2 \times \T_{H'}T) \cap \supp(\vphi_*(f))) \\
                                                    &\subset \pr_T(K \times K \times r^{-1}(C) \cap \supp(\vphi_*(f))) \\
                                                    &\subset \pr_T(r^{-1}(K \times C) \cap \supp(\vphi_*(f))) \\
                                                    &\subset \pr_T(r_{\supp(\vphi_*(f))}^{-1}(K\times C)).
\end{align*}
Now, $K \times C \subset M \times \R$ is compact and $\supp(\vphi_*(f))$ is proper, hence $r_{\supp(\vphi_*(f))}^{-1}(K'\times C)$ is compact, and so is $r_{|\supp(\Theta_{\eta,\eta'}(\vphi_*(f)))}^{-1}(C)$ because it is closed. The same reasoning also works for the source map, and $\supp(\Theta_{\eta,\eta'}(\vphi_*(f)))$ is proper.\\
We then need to show that the distribution $\Theta_{\eta,\eta'}(\vphi_*(f))$ is a $\CCC^{\infty}$ function. In order to do this, we show that its wave front set is empty. Material on wave front sets can be found in \cite{guillemin1990geometric, hormander2015analysis}. The main result on wave front sets that we will use is equation $(3.6)$ p.332 of \cite{guillemin1990geometric}: $WF(\vphi_*(f)) \subset \diff\vphi_*(WF(f)) \cup N_{\vphi}$, where $N_{\vphi} = \{ [(\vphi(x),\ell)] \in T^*(\T^{hol}_HM) \ / \ x \in \T_HM , \ell \circ \diff_x\vphi = 0 \}$. Since $WF(f) = \emptyset$, we get $WF(\vphi_*(f)) \subset \left[(H^0)^* \right] = PT^*U$ (the wave front set is located at time 0 on the $U \times U$ part of $\T^{hol}_HM$). Then, $\eta * \eta' \otimes 1 \in \CCC^{\infty}(\T^{hol}_HM)$, so multiplying by it does not yield a bigger wave front set. Finally, since $\pr_T$ is a submersion, we have $N_{\pr_T} = \emptyset$ and so, $WF(\Theta_{\eta,\eta'}(\vphi_*(f))) \subset \pr_T(PT^*U) = \emptyset$ and $\Theta_{\eta,\eta'}(\vphi_*(f)) \in \CCC^{\infty}(\T_{H'}T)$.
\end{proof}

\begin{proof}[Proof of theorem \ref{InducedOP}]
Since $\vphi \circ r_{\T_HM} = r_{\T_H^{hol}M}$, $\vphi_*(\PP)$ is transverse to $r_{\T_H^{hol}M}$. In the same fashion, since $\pr_T$ intertwines the range maps of $\T_H^{hol}M$ and $\T_{H'}T$, then $\PP^{\, \prime}$ is transverse to $r_{\T_{H'}T}$.\\
For quasi-homogeneity, since $\vphi$ is equivariant with respect to the $\R^*_+$-actions on $\T_HM$ and $\T_H^{hol}M$, then for $\lambda > 0$:
\begin{align*}
\alpha_{\lambda*}^{(T)} \PP^{\, \prime} - \lambda^m \PP^{\, \prime} &= \alpha_{\lambda*}^{(T)}\Theta_{\eta,\eta'} (\vphi_*(\PP)) - \Theta_{\eta,\eta'}(\vphi_*(\lambda^m \PP)) \\
											   &= \Theta_{\eta,\eta'}\left(\alpha_{\lambda*}^{hol}(\vphi_*(\PP)) - \vphi_*(\lambda^m \PP)\right) \\
											   &= \Theta_{\eta,\eta'}\left(\vphi_*(\alpha_{\lambda}(\PP) - \lambda^m \PP)\right).
\end{align*}
Since $\PP \in \bbpsi^m_H(M)$, we have $\alpha_{\lambda*}(\PP) - \lambda^m\PP \in \CCC^{\infty}_p(\T_HM)$. Hence by the last lemma $\alpha_{\lambda*}^{(T)}(\PP') - \lambda^m\PP' \in \CCC^{\infty}_p(\T_{H'}T)$, i.e. $\PP' \in \bbpsi^m_{H'}(T)$. Finally, we need to show that $P' = \PP^{\, \prime}_1$ is compactly supported. Computations similar to the previous ones give that:
$$\supp(\PP^{\, \prime}) \subset \pr_T(\supp(\eta\ast \eta'^* \otimes 1) \cap \vphi(\supp(\PP^{\, \prime})).$$
Since $\eta$ and $\eta'$ are compactly supported then so is $\eta\ast \eta'^*$. For a fixed $t \neq 0$, $\PP^{\, \prime}_t = \pr_{T*}((\eta\ast \eta'^* \otimes 1)\PP_t)$, hence $\supp(\PP^{\, \prime}_t)\cap \Delta_T \subset \pr_T(\supp(\eta\ast\eta'^*\otimes 1)\cap \Delta_M) $ is compact. Since $\PP^{\, \prime}$ is properly supported this implies that $\supp(\PP^{\, \prime}_t) \subset M \times M$ is compact. In particular, $P'$ is compactly supported, hence the claim holds true.
\end{proof}

\begin{cor}If $P \in \Psi^*_H(M)$ has symbol $\sigma \in \bar{S}^0(T_HM)$, then $T_{\eta'}^*PT_{\eta} \in \Psi^m_{H',c}(T)$ and has symbol 
$$(t,\xi) \mapsto \int_{u \in U}\eta\ast \eta'^*(u,u,t)\int_{H^0}\sigma(u,t,\xi).$$
\end{cor}
\begin{proof}
Take any $\PP \in \bbpsi^m_H(M)$ such that $\PP_1 = P$. Then $\PP^{\, \prime} := T_{\eta'}^*\vphi_*(\PP)T_{\eta} \in \bbpsi^m_{H'}(T)$ but because $\vphi_1 = \Id_{M\times M}$ we then have $\PP^{\, \prime}_1 = T_{\eta'}^*PT_{\eta} \in \Psi^m_{H'}(T)$. Let us now compute the symbol of $P'$. Take $\PP$ and $\PP^{\, \prime}$ as in theorem \ref{InducedOP}, then $\PP_0 = \sigma$ and $\PP^{\, \prime}_0 = T_{\eta'}^*\vphi_*(\sigma)T_{\eta}$. First we compute $\vphi_*(\sigma)$. Let $f\in \CCC^{\infty}_c(\ho(H^0)\ltimes T_{\sfrac{H}{H^0}}M  )$ and recall that here we identified $\ho(H^0)\ltimes T_{\sfrac{H}{H^0}}M   = \ho(H^0) \times_T T_{H'}T$:
\begin{align*}
\lag \vphi_*\sigma,f\rag &= \lag \sigma, f\circ \vphi \rag \\
					    &= \int_{(u,\delta,t,\xi) \in TU \times T_{H'}T} \sigma(u,\delta,t,\xi)f(u,u,t,\xi) \\
					    &=\int_{\Delta_U \times T_{H'}T} \left(\int_{H^0}\sigma\right)(u,t,\xi)f(u,u,t,\xi) \\
					    &=\int_{U\times U \times T_{H'}T} \left(\int_{H^0}\sigma\right)\delta_{\Delta_U}f \\
					    &=\lag \int_{H^0}\sigma \delta_{\Delta_U},f\rag.
\end{align*}
Hence $\vphi_*(\sigma) = \delta_{\Delta_U} \int_{H^0}\sigma$. We can now compute $\PP^{\, \prime}_{0}$, let $g \in \CCC^{\infty}(T_{H'}T)$:
\begin{align*}
\lag \PP'_{0} ,g\rag &= \lag \pr_{T*}\left( (\eta\ast\eta'^*\otimes 1) \int_{H^0}\sigma  \delta_{\Delta_U}\right),g \rag \\
					&= \lag \int_{H^0}\sigma \delta_{\Delta_U}, \eta'\ast\eta^* \otimes g \rag \\
					&= \int_{u\in U, t\in T,\xi \in T_{H',x}T}\eta\ast\eta'^*(u,u,t)\lag \int_{H^0}\sigma(u,t,\xi), g(t,\xi) \rag,
\end{align*} 
hence $\PP^{\, \prime}_{0}(t,\xi) = \int_{u \in U}\eta\ast \eta'(u,u,t)\int_{H^0}\sigma(u,t,\xi)$.
\end{proof}

The next proposition uses the language of connections on Hilbert modules. For background on connections and the proofs of the results used here, see the appendix of \cite{connes1984longitudinal}.

\begin{prop}\label{Connexion}Let $\sigma \in \bar{S}^0(T_HM)$ such that $\int_{H^0}\sigma = 0$. Let $P \in \Psi^*_H(M)$ be an operator with symbol $\sigma$, then $\forall F \in C^*(\ho(H^0)), P(F\otimes_{\CCC_0(T)}1), (F\otimes_{\CCC_0(T)}1)P \in \mathcal{K}(\mathscr{E})$, i.e. $P$ is a $0$-connection for $C^*(\ho(H^0))$.\\
Let $\sigma \in \bar{S}^0(T_HM)$ . Let us assume that the class modulo $C^*_M(T_{\sfrac{H}{H^0}}M)$ of $\int_{H^0}\sigma$  contains an element of the form $1\otimes \tilde{\sigma}$ with $\tilde{\sigma} \in \bar{S}^0(T_{H'}T)$. Let $P \in \Psi^*_H(M)$ and $\tilde{P} \in \Psi^*_{H'}(T)$ of respective symbols $\sigma$ and $\tilde{\sigma}$, then $P$ is a $\tilde{P}$-connection for $C^*(\ho(H^0))$.
\end{prop}

\begin{proof}
If $\int_{H^0}\sigma = 0$ then, with the same notations as before, $P'$ has a vanishing symbol of order $0$. We thus have $P' \in \Psi^{-1}_{H',c}(T)$. Being of negative order and compactly supported, $P'$ extends to a compact operator on $\mathscr{E}_{2|t=1}$. This works for every $\eta,\eta' \in \CCC^{\infty}_c(\ho(H^0))$, thus $P$ is a $0$-connection. If $F\in C^*(\ho(H^0))$ then $(F\otimes1)P, P(F\otimes1) \in \mathcal{K}(\mathscr{E})$.\\
Let us assume that $\sigma = 1 \otimes \tilde{\sigma}$. Then the previous corollary gives that $\PP'_0 = \lag\eta,\eta'\rag \tilde{\sigma}$, hence $P' - \lag \eta,\eta'\rag \tilde{P} \in \mathcal{K}(\mathcal{E}_{2|t=1})$ (they are compactly supported and have the same symbol). Then by \cite{connes1984longitudinal} remark A.6.4, $P$ is a $\tilde{P}$-connection.
\end{proof}

Let us now return to the general case where the filtration is not necessarily trivial.

\begin{thm}\label{PseudodiffCycle} Let $E \to M$ be an $\ho(H^0)$-equivariant, $\faktor{\Z}{2\Z}$-graded, hermitian bundle. Let $\sigma \in \Sigma^0(T_{\sfrac{H}{H^0}}M,E)$ be a $\ho(H^0)$-invariant odd symbol with $\sigma^2-1=0$. Let $P \in \Psi^*_H(M,E)$ be an operator with transverse symbol $\sigma$. Then $P$ induces a K-homology cycle $(\mathscr{E},P) \in \mathds{E}(C^*(\ho(H^0)),\mathcal{K}(L^2(M))) \cong \mathds{E}(C^*(\ho(H^0)),\CC)$.
\end{thm}

This theorem follows immediately from the following lemma:

\begin{lem}If $P \in \Psi^*_H(M,E)$ has vanishing transverse symbol then:
$$\forall F \in C^*(\ho(H^0)), P(F\otimes 1),(F\otimes 1)P \in \mathcal{K}(\mathscr{E}).$$
If $P \in \Psi^*_H(M,E)$ has a transverse symbol which is $\ho(H^0)$-invariant then: 
$$\forall F \in C^*(\ho(H^0)), [P,F\otimes 1] \in \mathcal{K}(\mathscr{E}).$$
\end{lem}
\begin{proof}
Using partitions of unity, we can see that $C^*(\ho(H^0))$ is generated by its restriction to foliated charts. We can thus restrict our study to $F\in C^*(\ho(H^0))$ having (compact) support in $\Omega\times_T\Omega$, where $U\times T = \Omega \subset M$ is a foliated chart. We want to use proposition \ref{Connexion} but $P$ does not necessarily have support included in $\Omega^2$. Let $\vphi\in \CCC^{\infty}_c(\Omega)$ be a cutoff function such that $\vphi F = F\vphi = F$, then:
\begin{align*}
P &=  P\vphi + P(1-\vphi) \\
  &=  P \vphi +  \vphi P (1-\vphi) + (1-\vphi)P(1-\vphi) \\
  &=  P\vphi(2-\vphi) - P\vphi(1-\vphi) + \vphi P (1-\vphi) + (1-\vphi)P(1-\vphi) \\
  &=  P\vphi (2-\vphi) + [\vphi,P](1-\vphi) + (1-\vphi) P (1-\vphi) \\
  &=  (P' + K) + K' + (1-\vphi) P (1-\vphi).
\end{align*}
Here, $P' \in \Psi^0_{H,c}(\Omega,E)$ has symbol $\vphi (2-\vphi) \sigma_H^0(P)$. $K = P\vphi(2-\vphi)-P'$ is compactly supported and has vanishing order zero symbol, it is then of order $-1$, hence compact, $K \in \mathcal{K}(\mathscr{E})$. $K' = [\vphi, P](1-\vphi)$ and since multiplication by $\vphi$ is an order $0$ operator then $K' \in \Psi^0_H(M,E)$. However, $\sigma_H^0(K') = [\vphi,\sigma_H^0(P)](1-\vphi) = 0$ since $\vphi$ is a multiplication by a scalar on each fiber of $T_HM$. Therefore $K' \in \Psi^{-1}_{H,c}(M,E)$ and is then compact: $K' \in \mathcal{K}(\mathscr{E})$. Finally, since $(1-\vphi) P (1-\vphi)$ and $F$ have disjoint support, we get $F (1-\vphi) P (1-\vphi) = (1-\vphi) P (1-\vphi) F = 0$. The first part of the lemma hence follows from proposition \ref{Connexion} applied to $P'$. For the second part, consider $\sigma_H^0(P)$ as a symbol on $\Omega$: $\sigma_H^0(P)_{|\Omega} \in \bar{S}^0(T_H\Omega,E)$. This symbol has a transverse part which is $\ho(H^0)$-invariant and is thus, modulo $C^*_M(T_{\sfrac{H}{H^0}\Omega},E)$, of the form $1\otimes \tilde{\sigma}$ with $\tilde{\sigma} \in \bar{S}^0(T_{H'}T,E')$. By proposition \ref{Connexion}, let $P_{\Omega} \in \Psi^*_H(\Omega,E)$ with symbol $\sigma_H^0(P)_{|\Omega}$, then $P_{\Omega}$ is a $\tilde{P}$-connection for $C^*(\ho(H^0))_{|\Omega}$. In particular $[P_{\Omega},F\otimes 1]$ is compact. Since $F$ commutes with $\vphi$ then $[\vphi(2-\vphi)P_{\Omega},F\otimes 1] = \vphi(2-\vphi)[\vphi(2-\vphi)P_{\Omega},F\otimes 1]$ is also compact. Finally $P'$ and $\vphi(2-\vphi)P_{\Omega}$ have the same symbol and are compactly supported hence their difference is compact. We thus get that $[P',F\otimes 1]$ is compact and so is $[P,F\otimes 1]$.
\end{proof}

\subsection{The link between the two constructions}

We now want to show that the two KK-classes constructed thus far are equal. We do this using the tangent groupoids. Let $E\to M$ be a vector bundle as before, and define $\mathscr{E} := C^*(\T_H^{hol}M,r^*E)$. Note that $\mathscr{E}_{|0}$ corresponds to the $C^*(\ho(H^0)\ltimes T_HM)$-Hilbert module used in the construction of $j_{\ho(H^0)}([\sigma])\in \kk(C^*(\ho(H^0)),C^*(\ho(H^0)\ltimes T_{\sfrac{H}{H^0}}M))$ (with $[\sigma]$ as in theorem \ref{KK-cycle-equiv}) and $\mathscr{E}_{|1}$ to the $C^*(M\times M)$-Hilbert module used in the construction of the KK-cycle $[P]\in \kk(C^*(\ho(H^0)),\CC)$ in theorem \ref{PseudodiffCycle}. The goal is hence to use some $\PP \in \bbpsi^0_H(M)$ and $\mathscr{E}$ to construct a KK-cycle $(\mathscr{E},\PP)$ linking the two previous cycles through the evaluation maps $\ev_0$ and $\ev_1$. 

\begin{lem}If $\PP \in \bbpsi^0_{H,c}(M,E)$, then $\PP$ extends to an element of $\mathcal{B}_{C^*(\T_HM_{|[0;1]})}(\mathscr{E}_{|[0;1]})$, and $\vphi_*\PP$ to an element of $\mathcal{B}_{C^*(\T_H^{hol}M_{|[0;1]})}(\mathscr{E}_{|[0;1]})$\end{lem}

\begin{proof}
As before, using a partition of unity, we can reduce to the case where $M = U\times T$. We then have $\T_H^{hol}M \cong \ho(H^0) \times_T \T_{H'}T$. We then need to show that $\PP$, being of order $0$, extends to a bounded operator on $C^*(\T_HM_{|[-1;1]}, r^*E)$. The proof is standard as it does not differ from the usual pseudodifferential calculus:  it follows from the estimates in \cite{van_erpyunken} Corollary 45 which can be taken uniformly for $t\in [-1;1]$. The proof then reduces to the fact that order 0 H-pseudodifferential operators extend to continuous operators on $L^2$-spaces (which uses the same estimates). \\
The foliated chart gives maps: 
$$\pr_{T*} \colon \mathcal{B}_{C^*(\T_HM)}(C^*(\T_HM,r^*E)) \to \mathcal{B}_{C^*(\T_{H'}T)}(C^*(\T_{H'}T,r^*E')),$$
coming from the submersion $\pr_T \colon \T_HM \to \T_{H'}T$ and:
$$j \colon \mathcal{B}_{C^*(\T_{H'}T)}(C^*(\T_{H'}T,r^*E')) \to \mathcal{B}_{C^*(\T_H^{hol}M)}(\mathscr{E}).$$
They are first defined on compact operators (coming from $\T_H^{hol}M = \ho(H^0)\ltimes \T_{H'}T$), and then extended to the multipliers algebra, i.e. the algebra of bounded operators. As all these morphisms restrict over $[0;1]$, we get the relation $\vphi_*\PP = j\circ \pr_{T*}\PP$, hence $\PP \in \mathcal{B}_{C^*(\T_H^{hol}M_{|[0;1]})}(\mathscr{E}_{|[0;1]})$
\end{proof}

We can thus extend $\bbpsi^0_H(M,E)_{|[0;1]}$ to a $C^*$-algebra $\bbpsi^0_{H,0}(M,E)$ by taking its closure inside $\mathcal{B}_{C^*(\T_HM_{|[0;1]})}(\mathscr{E}_{|[0;1]})$. We now have that $\ev_0(\bbpsi^0_H(M,E)) = \bar{S}^0_0(T_HM,E)$ and $\ev_1(\bbpsi^0_H(M,E)) = \Psi^0_{H,0}(M,E)$. As before, we extend this algebra to include operators bounded at infinity, namely:
$$\bbpsi^*_H(M,E) = \{\PP \in \mathcal{B}_{C^*(\T_HM_{|[0;1]})}(\mathscr{E}_{|[0;1]}) \ / \ \forall f \in \CCC_0([0;1]\times M), f\PP,\PP f \in \bbpsi^0_{H,0}(M,E)\}.$$

\begin{thm}\label{Link}Let $\sigma \in \Sigma^0(T_{\sfrac{H}{H^0}}M,E)$ be a bounded order $0$ transverse symbol, $\ho(H^0)$-invariant, with $\sigma^2=1$. Let $\PP \in \bbpsi^*_H(M,E)$ with $\int_{H^0}\PP_0$ representing $\sigma$. Then:
$$(\mathscr{E}_{|[0;1]},\PP) \in \mathds{E}(C^*(\ho(H^0)),C^*(\T^{hol}_HM_{|[0;1]})).$$
Moreover we have:
$$\ev_0([(\mathscr{E}_{|[0;1]},\PP)]) = j_{\ho(H^0)}([\sigma]) \textrm{ and } \ev_1([(\mathscr{E}_{|[0;1]},\PP)]) = [(\mathscr{E}_{|1},\PP_1)].$$
In particular, if the foliation is amenable, then $[(\mathscr{E}_{|1},\PP_1)] = \ind_H^{hol} \otimes j_{\ho(H^0)}([\sigma])$.
\end{thm}

\begin{proof}
Seeing $\mathscr{E}_{|[0;1]}$ as fibered over $[0;1]$, $\mathcal{B}_{C^*(\T_HM_{|[0;1]})}(C^*(\T_HM_{|[0;1]},r^*E))$ is identified with the algebra of continuous sections of a  $C^*$-bundle over the compact base $[0;1]$, with fiber at $t$ equal to $\mathcal{B}_{C^*(\T_HM_{|t})}(C^*(\T_HM_{|t},r^*E))$.\\
The algebra of compact operators then corresponds to the sections of the sub-$C^*$-bundle with fibers at $t$ equal to $\mathcal{K}_{C^*(\T_HM_{|t})}(C^*(\T_HM_{|t},r^*E))$. Because we already know that we are dealing with bounded operators, i.e. continuous sections, the compactness conditions of Kasparov modules can be checked on each fiber. For $t = 0$, this is the result of theorem \ref{OpSymbols}, and for $t \in (0;1]$, of theorem \ref{PseudodiffCycle}.\\
We then have $[(\mathscr{E}_{|[0;1]},\PP)]\otimes [ev_0] = j_{\ho(H^0)}(\sigma)$ and $[(\mathscr{E}_{|[0;1]},\PP)]\otimes [ev_1] = [(\mathscr{E}_{|1},P)]$, hence the result since $\ind_H^{hol} = [ev_0]^{-1} \otimes [ev_1]$
\end{proof}

\begin{rem}The pseudodifferential construction did not need amenability-like assuptions on the foliation, nor did theorem \ref{Link}. Indeed, let $\PP \in \Psi_H^0(M,E)$ be the extension of a transversally Rockland $P$ with symbol $\sigma$. In full generality, we have the classes $[\sigma] \in \kk^{\ho(H^0)}(\CCC_0(M),C^*(T_{\sfrac{H}{H^0}}M))$ from theorem \ref{KK-cycle-equiv}, $[P] \in \kk(C^*(\ho(H^0)),\CC)$ from theorem \ref{PseudodiffCycle} and $[\PP] \in \kk(C^*(\ho(H^0)),C^*(\T_H^{hol}M_{|[0;1]}))$. We also have the relations 
\begin{align*}
[\PP]\otimes [\ev_0] &= j_{\ho(H^0)}([\sigma]), \\
[\PP]\otimes [\ev_1] &= [P].
\end{align*}
\end{rem}

\begin{cor} In the case of a trivial filtration $H^1 =TM$, $j_{\ho(H^0)}([\sigma])\otimes \ind_H^{hol}$ is the class $\Psi(E,\sigma)$ constructed in \cite{hilsumskandalis} Theorem A.7.5.
\end{cor}
\begin{proof}
In the case of a trivial filtration, the cycle $(\mathscr{E}_{|1},P)$ is exactly the one of \cite{hilsumskandalis}.
\end{proof}

Let us now consider the case where the foliation is a fibration. Let us denote by $B$ its base space, and $\pi \colon M \to B$ the projection. The filtration $H^1 \subset \cdots \subset H^r = TM$ induces a filtration $H'^1 \subset \cdots \subset H'^r = TB$ by quotient (and through the identification $\faktor{TM}{H^0} \cong \pi^*TB$). The map $p$ is a Carnot map and we thus have a canonical isomorphism $T_{\sfrac{H}{H^0}}M \cong \pi^*(T_{H'}B)$.\\
Through this identification, a symbol $\sigma$ of order $k$ which is holonomy invariant defines a symbol $\pi_*(\sigma) \in S^k(T_{H'}B)$. It is Rockland if and only if $\sigma$ is transversally Rockland. Let us then consider a transversally Rockland symbol $\sigma \in \Sigma^0(T_{\sfrac{H}{H^0}}M)$ of order $0$, then $\pi_*(\sigma) \in \Sigma^0(T_{H'}B)$ is a Rockland symbol and hence defines a $\kk$-class $[\pi_*(\sigma)] \in \kk(\CCC_0(B), C^*(T_{H'}B)$.

\begin{thm}Under the isomorphism $\kk(C^*(\ho(H^0)),C^*(\ho(H^0)\ltimes T_{\sfrac{H}{H^0}}M)) \cong \kk(\CCC_0(B),C^*(T_{H'}B)$, we have:
$$j_{\ho(H^0)}([\sigma]) = [\pi_*(\sigma)].$$
Moreover the identification 
$$\kk(C^*(\ho(H^0)\ltimes T_{\sfrac{H}{H^0}}M),C^*(\T_{H}^{hol}M_{|[0;1]})) \cong\kk(C^*(T_{H'}B,C^*(\T_{H'}B_{|[0;1]})),$$
sends $\ind_H^{hol}$ to the class $\ind_{H'}$. Consequently, if $P \in \Psi^*_H(M)$ has symbol $\sigma$, then $\pi_*(P) \in \Psi^*_{H'}(B)$ has symbol $\pi_*(\sigma)$ and, under the isomorphism between the $K$-homology groups of $C^*(\ho(H^0))$ and $B$, we get $[P] = [\pi_*(P)]$.
\end{thm}
\begin{proof}
Let us first extend $\pi$ to a groupoid morphism $\pi \colon \T_HM \to \T_{H'}B$: over a trivializing subset of the fibration, we can write $M$ as $F \times B$. Then $\T_HM \cong \T F \times_{\R} \T_{H'}B$ and the morphism is the projection on the second factor. Let $\PP \in \bbpsi^0_{H}(M,E)$ such that $\PP_0 = \sigma$. Since $\pi$ is a surjective submersion, $\pi_*(\CCC^{\infty}_p(\T_HM)) \subset \CCC^{\infty}_p(\T_{H'}B)$ and because $\pi$ is equivariant with respect to the $\R^*_+$-actions on $\T_HM$ and $\T_{H'}B$, we obtain that $\pi_*\PP \in \bbpsi^0_{H'}(B,E')$. The class $[\pi_*(\sigma)]$ is then represented by $(\mathscr{E}', \pi_*\PP_1')$, which is exactly the image of $(\mathscr{E},\PP_1)$ under the Morita equivalence between $C^*(\ho(H^0))$ and $\CCC_0(B)$.
\end{proof}

\begin{rem}\label{VersionGroupoide}We have assumed that the filtrations are 'full' in the sense that there exist an integer $r \geq 0 $ such that $H^r = TM$. If it is not the case then the filtration will still eventually become stationary. Denote by $r$ the integer from which the filtration becomes stationary. Then, the bracket condition on Lie filtrations forces $H^r$ to be an integrable subbundle, i.e. a foliation. The calculus developped here would then correspond to symbols and operators longitudinal to $H^r$ and transverse to $H^0$ and the deformation groupoids would have to be replaced by adiabatic deformations of $\ho(H^r)$ (the same constructions work thanks to the functoriality of those in \cite{mohsen2018deformation}).\\
More generally, we can extend our setup to an arbitrary Lie groupoid $G$ with a filtration of its algebroid:
$$\mathcal{A}^0G \subset \mathcal{A}^1G \subset \cdots \subset \mathcal{A}^rG = \mathcal{A}G,$$
with the bracket condition $\forall i,j\geq 0 \left[\Gamma(\mathcal{A}^iG), \Gamma(\mathcal{A}^jG) \right]\subset \Gamma(\mathcal{A}^{i+j}G)$. The calculus obtained is a $G$-filtered calculus and the addition of $\mathcal{A}^0G$ would then allow to define a transverse symbols and a transversal Rockland condition for these operators. The reasoning would be the same as what has been done here so, for the sake of comprehensiveness and not having too heavy notations, we restricted ourselves to the case of $G = M\times M$ and a 'full' filtration. \end{rem}

\subsection{A link with equivariant index for countable discrete group actions}

Let $M$ be a filtered manifold with filtration $H^1 \subset \cdots \subset H^r = TM$ and $\Gamma$ be a countable discrete group acting on $M$ by filtration preserving diffeomorphims. Let $P$ be a manifold with fundamental group $\Gamma$ \footnote{This is always possible by taking open subsets in $\R^5$. If the group is finitely presented we can even assume $P$ to be compact.}. Let $\tilde{P}\to P$ be the universal cover of $P$. The manifold $\tilde{X} = \tilde{P} \times M$ is foliated by the fibration $\pr_M \colon \tilde{X} \to M$ (i.e. the leaves are the $\tilde{P}\times \{m\}$ for $m \in M$), let us denote by $\tilde{H}^0 = T\tilde{P}\times M$ the corresponding integrable subbundle. The filtration on $M$ induces a filtration on $\tilde{X}$ with $\tilde{H}^i = T\tilde{P}\times H^i$ for $i \geq 1$, so that $(\tilde{H}^i)_{i\geq 0}$ is a foliated filtration on $\tilde{X}$. The group $\Gamma$ acts diagonally on $\tilde{X}$ and preserves the foliation and the filtration. Moreover, since the action $\Gamma \act \tilde{P}$ is free and proper, then so is the diagonal action $\Gamma \act \tilde{X}$. Let $X = \tilde{P}\times_{\Gamma}M$ be the quotient manifold. Since the foliated filtration was $\Gamma$-invariant, it induces a foliated filtration $(\bar{H}^i)_{i\geq 0}$ on $X$. We have $T_{\tilde{H}}\tilde{X} = TX \times T_HM$ thus $T_{\sfrac{\tilde{H}}{\tilde{H}^0}}\tilde{X} \cong \tilde{P} \times T_HM$. From this we deduce that, for any $m \in M$, the projection $\tilde{X} \to X$ induces an isomorphism $T_{H,m}M \cong T_{\sfrac{\bar{H}}{\bar{H}^0},[p,m]}X$. The resulting foliation is made up to replace the group action. More precisely, since the foliation on $\tilde{X}$ is given by the projection on $M$, then $\ho(\tilde{H}^0)$ is Morita equivalent to $M$. From this, we get the Morita equivalence $\ho(\bar{H}^0)\sim \Gamma \ltimes M$.

Let us now consider $\sigma \in \Sigma(T_HM;E)^{\Gamma}$, an equivariant symbol on an equivariant hermitian bundle $E$. The pullback by $\pr_M$ induces the commutative diagram (in the category of $\Gamma$-$C^*$-algebras):
$$\xymatrix{0 \ar[r] & C^*_{M}(T_HM;E) \ar[r] \ar[d]^{\pr_M^*} & \bar{S}^0(T_HM;E) \ar[r] \ar[d]^{\pr_M^*} & \Sigma^0(T_HM;E) \ar[r] \ar[d]^{\pr_M^*} & 0 \\
0 \ar[r] & C^*_{\tilde{X}}(T_{\sfrac{\tilde{H}}{\tilde{H^0}}}\tilde{X};\pr_M^*E) \ar[r] & \bar{S}^0(T_{\sfrac{\tilde{H}}{\tilde{H^0}}}\tilde{X};\pr_M^*E) \ar[r] & \Sigma^0(T_{\sfrac{\tilde{H}}{\tilde{H^0}}}\tilde{X};\pr_M^*E) \ar[r] & 0.}$$
Let us define $\tilde{\sigma} \in \Sigma^0(T_{\sfrac{\tilde{H}}{\tilde{H^0}}}\tilde{X};\pr_M ^*E)$ as $\pr_M^*(\sigma)$, it is equivariant with respect to the diagonal action. If $q \colon \tilde{X} \to X$ denotes the canonical projection, then we get the same type of diagrams as before:
$$\xymatrix{0 \ar[r] & C^*_{X}(T_{\sfrac{\bar{H}}{\bar{H^0}}}X;\bar{E}) \ar[r] \ar[d]^{q^*} & \bar{S}^0(T_{\sfrac{\bar{H}}{\bar{H^0}}}X;\bar{E}) \ar[r] \ar[d]^{q^*} & \Sigma^0(T_{\sfrac{\bar{H}}{\bar{H^0}}}X;\bar{E}) \ar[r] \ar[d]^{q^*} & 0 \\
0 \ar[r] & C^*_{\tilde{X}}(T_{\sfrac{\tilde{H}}{\tilde{H^0}}}\tilde{X};\pr_M^*E) \ar[r] & \bar{S}^0(T_{\sfrac{\tilde{H}}{\tilde{H^0}}}\tilde{X};\pr_M^*E) \ar[r] & \Sigma^0(T_{\sfrac{\tilde{H}}{\tilde{H^0}}}\tilde{X};\pr_M^*E) \ar[r] & 0.}$$
Moreover, should we replace all the algebras on the bottom line by the fixed points algebras for the $\Gamma$-action, then every vertical arrow would become an isomorphism. The symbol $\tilde{\sigma}$ thus factors into a symbol $\bar{\sigma} \in \Sigma^0(T_{\faktor{\bar{H}}{\bar{H^0}}}X,\bar{E})$, where $\bar{E}$ corresponds to the quotient bundle of $\pr_M^*E$ on $X$. Because of the previous isomorphims for the fibers of $T_{\sfrac{\bar{H}}{\bar{H}^0}}X$ and the construction of $\bar{\sigma}$, we get that $\bar{\sigma}$ is transversally Rockland if and only if $\sigma$ is Rockland. Let us assume $\sigma^2 = 1$ hence $\bar{\sigma}^2 = 1$. We then are under the conditions of theorem \ref{KK-cycle-equiv} and get a class:
$$[\bar{\sigma}] \in \kk^{\ho(\bar{H^0})}(\CCC_0(X),C^*(T_{\sfrac{\bar{H}}{\bar{H^0}}}X)).$$
One could also directly see that $\sigma$ itself defines an equivariant class:
$$[\sigma] \in \kk^{\Gamma \ltimes M}(\CCC_0(M),C^*(T_HM)).$$
The Morita equivalence $\ho(\bar{H^0}) \sim \Gamma \ltimes M$ sends $\CCC_0(X)$ to $\CCC_0(M)$ and $C^*(T_{\sfrac{\bar{H}}{\bar{H^0}}}X)$ to $C^*(T_HM)$, giving an isomorphism:
$$\kk^{\ho(\bar{H^0})}(\CCC_0(X),C^*(T_{\sfrac{\bar{H}}{\bar{H^0}}}X)) \cong \kk^{\Gamma \ltimes M}(\CCC_0(M),C^*(T_HM)),$$
sending $[\bar{\sigma}]$ to $[\sigma]$.

\begin{rem}The previous sections can be applied to prove the following: let $P \in \Psi^*_H(M;E)$ be an operator on $M$ with symbol $\sigma$, then $P$ is equivariant modulo compact operators and defines a $\kk$-class $[P] \in \kk^{\Gamma}(\CCC_0(M),\CC)$. Applying Thom isomorphism twice and Kasparov's second Poincaré duality theorem \cite{KasparovEquiv}, we get an isomorphism $\kk^{\Gamma \ltimes M}(\CCC_0(M),C^*(T_HM)) \cong \kk^{\Gamma}(\CCC_0(M),\CC)$. It would be interesting to show how $[\sigma]$ and $[P]$ are intertwined through this isomorphism.
\end{rem}

\nocite{*}
\bibliographystyle{plain}
\bibliography{Ref}

\end{document}